\documentclass[reqno,centertags,11pt,a4paper]{amsart}
\usepackage{amssymb}
\usepackage{mathrsfs}
\usepackage{amsmath, amsfonts,vmargin, enumerate}
\usepackage{color}
\usepackage{amsthm}
\usepackage{latexsym,bm}
\usepackage{euscript}
\usepackage{color}
\usepackage{stmaryrd}
\usepackage[colorlinks=true,citecolor=blue,linkcolor=blue]{hyperref}

\begin{document}
	\newtheorem{theorem}{Theorem}
	\newtheorem{proposition}[theorem]{Proposition}
	\newtheorem{conjecture}[theorem]{Conjecture}
	\newtheorem{corollary}[theorem]{Corollary}
	\newtheorem{lemma}[theorem]{Lemma}
	\newtheorem{sublemma}[theorem]{Sublemma}
	\newtheorem{observation}[theorem]{Observation}
	\newtheorem{remark}[theorem]{Remark}
	\newtheorem{definition}[theorem]{Definition}
	\theoremstyle{definition}
	\newtheorem{notation}[theorem]{Notation}
	\newtheorem{question}[theorem]{Question}
	\newtheorem{example}[theorem]{Example}
	\newtheorem{problem}[theorem]{Problem}
	\newtheorem{exercise}[theorem]{Exercise}
	\numberwithin{theorem}{section} 
	\numberwithin{equation}{section}
	
	\title[Dimension-free estimate for  semi-commutative discrete spherical maximal operator]{Dimension-free estimate for  semi-commutative discrete spherical maximal operator}

	\author{Yue Zhang}
	\address{Institute for Advanced Study in Mathematics, Harbin Institute of Technology, 150001 Harbin, China}
	\email{21b912030@stu.hit.edu.cn}
	
	%

\date{}
\maketitle

\markboth{Y. Zhang}
{Semi-commutative discrete spherical maximal operator}
	\begin{abstract}
	In this paper, we establish  dimension-free estimates for the discrete spherical maximal operator  on semi-commutative $L_{p}$ space for $2\leq p\leq\infty$.
	\end{abstract}
	\maketitle
	\section{Introduction}

Modern discrete harmonic analysis originates from Bourgain's work on the pointwise ergodic theorem along polynomial orbits (see \cite{bourgain1998-1,bourgain1998-2}). Since then, numerous research papers  have been  explored in this area (see e.g.   \cite{MR3992568, MR4175747, Hughes,Ionescu, Ionescu-W, Kesler-R, Kesler2}). The dimension-free estimate of  the discrete spherical maximal operator is one of the topics. We recall the result here. Let $\mathbb{S}=\left\{x\in\mathbb{R}^d: |x|=1\right\}$ be the unit sphere  in $\mathbb{R}^d$ and   $\mathbb{I}\subset\mathbb{R}_{+}$  be a non-empty index set. For $t\in\mathbb{I}$ and $x\in\mathbb{Z}^d$, we denote the discrete spherical  average by
\begin{align}\label{250402.3}
	\mathcal{A}^{d}_{t}f(x)=\frac{1}{|\mathbb{S}_{t}\cap \mathbb{Z}^d|}\sum_{y\in \mathbb{S}_{t}\cap \mathbb{Z}^d} f(x-y),\quad f\in \ell_{1}(\mathbb{Z}^d),
\end{align}
 where $\mathbb{S}_{t}=\left\{x\in\mathbb{R}^d: |x|=t\right\}$ is the dilation of the unit sphere  $\mathbb{S}$. The   maximal spherical operator is given by
 \begin{align*}
 \mathcal{A}^{d}_{*}f(x)=\sup_{t\in \mathbb{I}}|\mathcal{A}^{d}_{t}f(x)|.
 \end{align*}
 Magyar  first established  a local spherical maximal inequality in \cite{Magyar1997}. 
 Later, Magyar, Stein, and Wainger (see \cite{Magyar})  studied   the case $\mathbb{I}= \sqrt{\mathbb{N}}=\left\{\lambda\in(0,\infty),\lambda^2\in \mathbb{N}\right\}$, and proved that if  $d\geq 5$ and $p>\frac{d}{d-2}$, there exists a constant $C_{d,p}>0$ depending  on  $d$ and $p$ such that  
\begin{align}\label{250519.1} 
	\| \mathcal{A}^{d}_{*}f\|_{\ell_{p}(\mathbb{Z}^d)}\leq  C_{d,p}\|f\|_{\ell_{p}(\mathbb{Z}^d)}. 
\end{align}
Moreover, they showed that the above ranges of $p$ and $d$ are optimal by constructing  counterexamples.  A natural question is whether the constant  $C_{d,p}$ in (\ref{250519.1}) is independent of  the dimension $d$.  This was affirmatively answered by Mirek et al. \cite{mirek2023dimension} in the case $2\leq p\leq\infty$ and $\mathbb{I}=\mathbb{D}=\left\{ 2^{n}: n\in \mathbb{N} \cup \left\{0\right\} \right\}$, who showed that  
\begin{align}\label{250603.1}
	\|\mathcal{A}^{d}_{*}f\|_{\ell_{p}(\mathbb{Z}^d)}\leq  C_{p}\|f\|_{\ell_{p}(\mathbb{Z}^d)},
\end{align} 
where  $C_{p}$ depends only on   $p$.

 

 Noncommutative harmonic analysis is a rapidly developing field, based on  harmonic analysis, operator algebras, noncommutative geometry, and quantum probability  (see e.g. \cite{MR3079331,MR4320770,MR4048299,MR3283931,mei2007operator,xia2016characterizations,MR3778570}).  
Among its branches, semi-commutative harmonic analysis offers a framework that is both accessible and   nontrivial, and it requires the development of new ideas.  For instance, let $\mathcal{M}$ be a von Neumann algebra and $f:\mathbb{R}^d\rightarrow \mathcal{ {M}}$ an operator-valued function belonging to   the semi-commutative $L_{p}$ space. The continuous spherical maximal operator for a sequence of operators
\begin{align*} 
	\mathbf{A}^{d}_{t}f(x)=\frac{1}{|\mathbb{S}_{t}|}\int_{\mathbb{S}_{t}}f(x-y) dy 
\end{align*}
 seems no available as   any two operators in $\mathcal{M}$ can not be compared directly. 
 Thus, establishing the corresponding   $L_{p}$   maximal inequality is much more subtle. This obstacle was not overcome until the introduction of Pisier’s vector-valued noncommutative   spaces $L_{p}(L_{\infty}(\mathbb{R}^d)\overline{\otimes} \mathcal{M};\,\ell_{\infty})$. Following Pisier’s approach, Hong   derived  the    semi-commutative spherical maximal   inequality for $d\geq 3$ and $p>
 \frac{d}{d-1}$ (see \cite[Proposition 4.1]{MR3275742}):
\begin{align*} 
	\|\sup_{t>0} \mathbf{A}^{d}_{t}f\|_{L_{p}(L_{\infty}(\mathbb{R}^d)\overline{\otimes}\mathcal{M})}\leq C_{p,d} \|f\|_{L_{p}(L_{\infty}(\mathbb{R}^d)\overline{\otimes}\mathcal{M})}.
\end{align*}
 Here, the  norm $\|\sup_{t>0} \mathbf{A}^{d}_{t}f\|_{L_{p}(L_{\infty}(\mathbb{R}^d)\overline{\otimes}\mathcal{M})}$  is understood as $L_{p}(L_{\infty}(\mathbb{R}^d)\overline{\otimes} \mathcal{M};\,\ell_{\infty})$ norm of  the sequence $( \mathbf{A}^{d}_{t}f)_{t>0}$ in the noncommutative case; we  refer  the reader  to Definition \ref{24926.1} in Section \ref{1222.1} for more details. Furthermore, Hong  studied the maximal ergodic inequality of spheres over Euclidean spaces  (see \cite{MR4048299}). Recently,   Chen  and Hong \cite{chen-hong}  paid attention to the discrete spherical maximal inequality and successfully transferred  Magyar-Stein-Wainger inequality (\ref{250519.1})   to the semi-commutative setting. Despite these
advancements,    dimension-free estimates for the  semi-commutative discrete maximal operator has not been explored.    In this paper, we  address this issue by investigating the semi-commutative analogue of inequality (\ref{250603.1}).

   Let $\mathcal{M}$ be a von Neumann algebra equipped with a normal semifinite faithful trace $\tau$.  The noncommutative $L_{p}$ space associated with $(\mathcal{{M}}, \tau)$ is denoted by $L_{p}(\mathcal{ {M}})$. Consider the tensor von Neumann algebra ${\mathcal{N}}=\ell_{\infty}(\mathbb{Z}^{d})\overline{\otimes} \mathcal{M}$ equipped  with the  trace $\sum\otimes\tau $.  The semi-commutative  space   associated with $(\mathcal{{N}}, \sum\otimes\tau)$ is denoted $L_{p}({\mathcal{N}})$ ,  which  coincides with $L_{p}\big(\mathbb{Z}^d;L_{p}(\mathcal{M})\big)$, the  $p$-integrable  functions from $\mathbb{Z}^d$ to $L_{p}(\mathcal{M})$. Given $t\in\mathbb{D}=\left\{ 2^{n}: n\in \mathbb{N} \cup \left\{0\right\} \right\}$, the semi-commutative discrete spherical averaging operator  $\mathcal{A}^{d}_{t}$ has the same form with (\ref{250402.3}) but acts on functions in $L_{p}(\mathcal{N})$. We establish the following semi-commutative maximal inequality.

\begin{theorem}\label{0109.5}
	Let $d\geq 16$ and  $f\in L_{p}(\mathcal{N})$ with $2\leq p\leq\infty$. There exists a constant $C_{p}>0$ independent of the dimension $d$ such that  
	\begin{align}\label{250606.1}
		\|\sup_{t\in \mathbb{D}}\mathcal{A}_{t}^{d}f\|_{L_{p}(\mathcal{N})}\leq C_{p}\|f\|_{L_{p}(\mathcal{N})}.
	\end{align}
\end{theorem}

 Theorem \ref{0109.5} extends the result of Mirek et al. \cite{mirek2023dimension} to the semi-commutative setting. To get (\ref{250606.1}), we need numerous improvement and modifications based on the main idea of \cite{mirek2023dimension} to overcome the difficulties due to noncommutativity. We first apply the noncommutative complex interpolation theorem to reduce (\ref{250606.1}) to the case $p=2$. This combined with the noncommutative transference principle and  dimension-free estimate for the semigroup $P_{t}$ on $\mathbb{Z}^d$ concludes the proof.
  
  
	The rest of paper is organized as follows.  In section \ref{1222.1}, we review   preliminaries on the noncommutative $L_{p}$ space and the vector-valued noncommutative $L_{p}$ space. In section  \ref{250414.1},   we  introduce  two useful tools:  dimension-free estimate for the semigroup $P_{t}$ on $\mathbb{Z}^d$, and  noncommutative analogue corresponding to the sampling principle in \cite[Corollary 2.1]{Magyar}. Section \ref{0118.4} is devoted to proving Theorem \ref{0109.5}, more precisely, the  dimension-free estimate of the semi-commutative discrete  spherical maximal operator on $L_{p}(\mathcal{N})$ for $2\leq p\leq \infty$.

\vspace{0.1cm}

\textbf{Notation}
  The letter  $C$  denotes a positive constant independent of the variables, while it may vary from line to line.  $A \lesssim B$ means $A \leq C B$ for some absolute  constant $C$ and $A \lesssim^d B$ means $A \leq C^d B$ for some absolute  constant $C$.  Given $N \in \mathbb{N}$, we set $\mathbb{N}_{N} =\left\{1,2,\cdots,N\right\}$ and $\mathbb{N}_{N}^{d}=\mathbb{N}_{N}\times\cdots\times \mathbb{N}_{N}$ with $d$-tuples product. The floor function $\lfloor t\rfloor=\max\left\{n\in\mathbb{Z}:n\leq t\right\}$ denotes the integer part of $t$. Given  a set $E\subset\mathbb{R}^d$,  its characteristic function is denoted by  $\chi_{E}$. Notation LHS (resp. RHS) stands for left hand side (rsep. right hand side ) of an expression. The symbol  $\sigma $ represents  the canonical surface measure on the unit sphere $\mathbb{S} $, and let
  \begin{align}\label{250413.1}
  	\mu  =\frac{1}{\sigma (\mathbb{S} )}
  \end{align}
  be its  normalization.  The torus $\mathbb{T}^d$ is a priori endowed with the periodic norm: 
 \begin{align*}
 	||x||= \Big(\sum_{j=1}^d ||x_j||^2\Big)^{1/2},
 \end{align*}
 where $||x_j|| = \operatorname{dist}(x_j, \mathbb{Z})$ for $j \in \mathbb{N}_d$.  We identify the $d$-dimensional torus $\mathbb{T}^d$  with the unit cube  $Q = [-{1}/{2},  {1}/{2})^d$. Observe that for $\xi \in Q$,  $||\xi||$ coincides with the Euclidean norm $|\xi|$.  Furthermore, we obtain
 \begin{align}\label{250424.5}
 	2\|\eta\|\leq|\sin(\pi\eta)|\leq \pi\|\eta\|,\quad \eta \in \mathbb{T},
 \end{align}  
 since $|\sin(\pi\eta)|=\sin(\pi\|\eta\|)$ and $2|\eta|\leq |\sin(\pi\eta)|\leq \pi\eta,\mbox{ for } 0\leq|\eta|\leq {1}/{2}.$ For $x \in \mathbb{R}^d$,  let $\llbracket x \rrbracket $ be the unique vector in $\mathbb{Z}^d$ such that $x - \llbracket x \rrbracket \in Q$. In particular, for $\xi \in Q$, one has $\llbracket \xi \rrbracket = 0$. Let $f$ be a function on $\mathbb{Z}^d$, its   Fourier transform   $\widehat{f}$  is given by:
    \begin{align*}
     \widehat{f}(\xi)=\sum_{x\in\mathbb{Z}^d}f(x)e^{-2\pi \mathrm{i} \langle x,\xi \rangle},\quad \xi\in\mathbb{T}^d.
    \end{align*}
    For a function $f$ on $\mathbb{T}^d$, its inverse Fourier transform   $\mathcal{F}^{-1}f$   is defined by:
    \begin{align*}
    	(\mathcal{F}^{-1}f)(x)=\int_{\mathbb{T}^{d}}f(\xi)e^{2\pi \mathrm{i} \langle x,\xi \rangle}d\xi,\quad x\in\mathbb{Z}^d.
    \end{align*}

\section{Preliminaries}\label{1222.1}
\textbf{2.1. Noncommutative $L_{p}$ space}. Let $\mathcal{M}$ be a von Neumann algebra equipped with a normal semifinite faithful trace $\tau$  and $\mathcal{M}_{+}$ be the positive part of $\mathcal{M}$.  Denoted by $\mathcal{S}_{+}(\mathcal{M})$ the set of     $x\in\mathcal{M}_{+}$ with  $\tau(s(x)) <\infty$, where $s(x)$ is the smallest projection $e$ satisfying $exe=x$.  $\mathcal{S}(\mathcal{M})$ is the linear span of $\mathcal{S}_{+}(\mathcal{M})$. Given $1\leq p<\infty$, we
define
\begin{align*}
	\|x\|_{L_{p}(\mathcal{M})}=\big(\tau(|x|^p)\big)^{{1}/{p}},\quad x\in \mathcal{S}(\mathcal{M}),
\end{align*}
where $|x|=(x^{*}x)^{{1}/{2}}$ is the modulus of $x$. Then $(\mathcal{S}(\mathcal{M}),\|\cdot\|_{L_{p}(\mathcal{M})})$ is a normed space, whose
completion is the noncommutative $L_{p}$ space associated with $(\mathcal{M}, \tau )$, denoted  
by $L_{p}(\mathcal{M})$.  Notation $L_{p}(\mathcal{M})_{+}$ is the positive part of $L_{p}(\mathcal{M})$. For convenience, we set $L_{\infty}(\mathcal{M})=\mathcal{M}$ equipped with the operator norm $\|\cdot\|_{\infty}$.

	In this paper, we are interested in   the tensor von
Neumann algebra $\mathcal{N}=\ell_{\infty}(\mathbb{Z}^d)\overline{\otimes}\mathcal{M}$ equipped with the  trace $\sum\otimes\tau$. Define $L_{p}({\mathcal{N}})$ as the semi-commutative  space   associated with $(\mathcal{{N}}, \sum\otimes\tau)$,  which  isometrically coincides with $L_{p}(\mathbb{Z}^d;L_{p}(\mathcal{M}))$,  the Bochner  $L_{p}$ space on $\mathbb{Z}^d$  with values in $L_{p}(\mathcal{M})$. Given $f\in L_{2}(\mathcal{N})$, the   operator-valued version of the Plancherel formula 
\begin{align}\label{250609.1}
	 \|f\|_{L_{2}(\mathcal{N})}=\|\widehat{f}\,\|_{L_{2}(  L_{\infty}(\mathbb{T}^d)\overline{\otimes}\mathcal{M})}, 
\end{align}
is essential for us, which is a consequence of the fact that  $L_{2}(\mathcal{N})$ is a Hilbert space.

\noindent\textbf{2.2. Noncommutative maximal functions}. A fundamental object  of this paper is the vector-valued noncommutative $L_{p}$ space $L_{p}(\mathcal{M};\,\ell_{\infty})$ introduced by Pisier \cite{MR1648908} and Junge \cite{MR1916654}. 
	\begin{definition}\label{24926.1}
	Given $1\leq p\leq\infty$, $L_{p}(\mathcal{M};\,\ell_{\infty})$ is the space of all sequences $x=(x_{n})_{n\in\mathbb{Z}}$ in $L_{p}(\mathcal{M})$ which admit factorizations of the following form: there are $a,b\in L_{2p}(\mathcal{M}) $ and a bounded sequence $y=(y_{n})_{n\in\mathbb{Z}}\subset L_{\infty}(\mathcal{M})$ such that $x_{n}=ay_{n}b$ for all $n\in\mathbb{Z}$. The norm of $x$ in $L_{p}(\mathcal{M};\,\ell_{\infty})$ is given by
	\begin{align*}
		\|x\|_{L_{p}(\mathcal{M};\,\ell_{\infty})}=\inf\left\{\|a\|_{L_{2p}(\mathcal{M})}\sup_{ n\in\mathbb{Z}}\|y_{n}\|_{\infty}\|b\|_{L_{2p}(\mathcal{M})}\right\},
	\end{align*}
	where the infimum is taken over all factorizations of $x= (x_{n})_{n\in\mathbb{Z}}=(ay_{n}b)_{n\in\mathbb{Z}}$ as above.
\end{definition}
  It is well known that $L_{p}(\mathcal{M};\,\ell_{\infty})$ is a Banach space equipped with the norm $\|\cdot\|_{L_{p}(\mathcal{M};\,\ell_{\infty})}$.  
For a more intuitive notation,  $\|x\|_{L_{p}(\mathcal{M};\,\ell_{\infty})}$ is often  denoted   by $\|\sup x_{n}\|_{L_{p}(\mathcal{M})}$. We
should point out that $\sup x_{n}$ is just a notation and it does not make any sense in the noncommutative setting. 

 To get a better understanding on  $L_{p}(\mathcal{M};\,\ell_{\infty})$,  let us consider a sequence of selfadjoint operators $x=(x_{n})_{n\in \mathbb{Z}}$  in $L_{p}(\mathcal{M})$. It was shown in \cite[Remark 4]{MR3079331} that   $x\in L_{p}(\mathcal{M};\,\ell_{\infty})$
if and only if there is a positive operator $a\in L_{p}(\mathcal{M})_{+}$ such that $-a\leq x_{n}\leq a $  for all $n\in \mathbb{Z}$, and moreover,
\begin{align}\label{250718.1}
	\|\sup_{n\in \mathbb{Z}}x_{n}\|_{L_{p}(\mathcal{M})}=\inf\left\{\|a\|_{L_{p}(\mathcal{M})}:a\in L_{p}(\mathcal{M})_{+},\,\,-a\leq x_{n}\leq a,\,\, \forall \,n\,\in \mathbb{Z} \right\}.
\end{align}

	More generally, if  $\Lambda$  is an arbitrary index set,   $L_{p}(\mathcal{M};\,l_{\infty}(\Lambda))$ is defined by the space of all sequences $x=(x_{\lambda})_{\lambda\in \Lambda}$ in $L_{p}(\mathcal{M})$ which admit   factorizations of the following form: there are $a,b\in L_{2p}(\mathcal{M}) $ and a bounded sequence $y=(y_{\lambda})_{\lambda\in \Lambda}\subset L_{\infty}(\mathcal{M})$ such that $x_{\lambda}=ay_{\lambda}b$ for all $\lambda\in \Lambda$. The norm of $x$ in $L_{p}(\mathcal{M};\,l_{\infty}(\Lambda))$ is given by
\begin{align*}
	\|x\|_{L_{p}(\mathcal{M};\,l_{\infty}(\Lambda))}=\inf_{x_{\lambda}=ay_{\lambda}b}\left\{\|a\|_{L_{2p}(\mathcal{M})}\sup_{ \lambda\in \Lambda}\|y_{\lambda}\|_{\infty}\|b\|_{L_{2p}(\mathcal{M})}\right\}.
\end{align*} 
	If $x=(x_{\lambda})_{\lambda\in \Lambda}$ is a sequence of selfadjoint operators in $L_{p}(\mathcal{M})$, then $\|x\|_{L_{p}(\mathcal{M};\,l_{\infty}(\Lambda))}$ has the similar property as (\ref{250718.1}), i.e.,
\begin{align}\label{250427.2} 
	\|x\|_{L_{p}(\mathcal{M};\,l_{\infty}(\Lambda))}=\inf\left\{\|a\|_{L_{p}(\mathcal{M})}:~a\in L_{p}(\mathcal{M})_{+},~\-a \leq x_{\lambda}\leq a, ~\forall \,\lambda\in \Lambda \right\}.
\end{align}

The space $L_{p}(\mathcal{M};\,\ell_{\infty})$  behaves well with respect to complex interpolation. 

\begin{theorem}[\cite{MR2276775}]\label{0.4}
	Let  $1\leq p_{0}<p< p_{1}\leq\infty$ and $ 0<\theta<1$ be such that $ {1}/{p}= {(1-\theta)}/{p_{0}}+ {\theta}/{p_{1}}$. The noncommutative complex interpolation holds
	\begin{align*}
		L_{p}(\mathcal{M};\,\ell_{\infty})=\big(L_{p_{0}}(\mathcal{M};\,\ell_{\infty}),L_{p_{1}}(\mathcal{M};\,\ell_{\infty})\big)_{\theta}\quad\mbox{with equal norms}.
	\end{align*}
\end{theorem}

 The following easy facts are  used for further study; we prove them here for completeness. 
\begin{lemma}\label{0109.1}
	Let $1\leq p\leq\infty$ and $f=(f_{n})_{n}\in L_{p}(\mathcal{{N}};\,\ell_{\infty})$. 
	\begin{itemize}
		\item[(i)] For any fixed $t\in\mathbb{R}^d$,   we have 
		\begin{align*}
			\|\sup_{n}e^{2\pi \mathrm{i}\langle t,\cdot\rangle}f_{n}(\cdot)\|_{L_{p}(\mathcal{N})}=	\|\sup_{n}f_{n}\|_{L_{p}(\mathcal{N})}.
		\end{align*}
		\item[(ii)] Given a bounded sequence  $(\beta_{n})_{n}\in\mathbb{C}$,  we have 
			\begin{align*}
				\|\sup_{n}\beta_{n}f_{n}\|_{L_{p}(\mathcal{N})}\leq 	\sup_{n}|\beta_{n}|\cdot\|\sup_{n}f_{n}\|_{L_{p}(\mathcal{N})}.
			\end{align*}
		\end{itemize}
\end{lemma}
\begin{proof}
	(i) Since  $f=(f_{n})_{n}\in L_{p}(\mathcal{{N}};\ell_{\infty})$,  for any $\delta>0$, there exist $a,b\in L_{2p}(\mathcal{N})$ and $(y_{n})_{n}\subset L_{\infty}(\mathcal{N})$ such that $f_{n}=ay_{n}b$  for all $n$, with 
\begin{align*}
	\|a\|_{L_{2p}(\mathcal{N})}\sup_{n}\|y_{n}\|_{\infty}\|b\|_{L_{2p}(\mathcal{N})}\leq \|\sup_{n}f_{n}\|_{L_{p}(\mathcal{N})}+\delta. 
\end{align*}
Considering the sequence  $(e^{2\pi \mathrm{i}\langle t,\cdot\rangle}f_{n}(\cdot))_{n}$,  each term admits the decomposition
\begin{align*}
	e^{2\pi \mathrm{i}\langle t,\cdot\rangle}f_{n}(\cdot)=e^{2\pi \mathrm{i}\langle t,\cdot\rangle}a(\cdot)y_{n}(\cdot)b(\cdot),\quad \forall n.
\end{align*} 
By Definition \ref{24926.1}, we get
\begin{align*}
	\|\sup_{n}e^{2\pi \mathrm{i}\langle t,\cdot\rangle}f_{n}(\cdot)\|_{L_{p}(\mathcal{N})}\leq&\|e^{2\pi \mathrm{i}\langle t,\cdot\rangle}a(\cdot)\|_{L_{2p}(\mathcal{N})}\sup_{n}\|y_{n}\|_{\infty}\|b\|_{L_{2p}(\mathcal{N})}\\
	=&\|a\|_{L_{2p}(\mathcal{N})}\sup_{n}\|y_{n}\|_{\infty}\|b\|_{L_{2p}(\mathcal{N})}\\
	\leq&\|\sup_{n}f_{n}\|_{L_{p}(\mathcal{N})}+\delta.
\end{align*}
The arbitrariness of $\delta$ implies that
\begin{align*}
	\|\sup_{n}e^{2\pi \mathrm{i}\langle t,\cdot\rangle}f_{n}(\cdot)\|_{L_{p}(\mathcal{N})}\leq \|\sup_{n}f_{n}\|_{L_{p}(\mathcal{N})}.
\end{align*}
The reverse inequality follows by the same argument.
	
	(ii) This proof is similar to (i). Hence, we omit the details here.
\end{proof}


\begin{lemma}\label{1203.4}
	Let $1\leq p<\infty$ and $f=(f_{n})_{n}\in L_{p}( L_{\infty}(\mathbb{R}^d)\overline{\otimes}\mathcal{M};\,\ell_{\infty})$.  For any $s>0$,    we have 
	\begin{align*}
		\|\sup_{n}f_{n} ({\cdot}/{s})\|_{L_{p}( L_{\infty}(\mathbb{R}^d)\overline{\otimes}\mathcal{M})}=s^{\frac{d}{p}}\|\sup_{n}f_{n}\|_{L_{p}( L_{\infty}(\mathbb{R}^d)\overline{\otimes}\mathcal{M})}.
	\end{align*}
\end{lemma}
\begin{proof}
	Since  $f\in L_{p}( L_{\infty}(\mathbb{R}^d)\overline{\otimes}\mathcal{M};\,\ell_{\infty})$,  for any $\delta>0$, there exist  $a,b\in L_{2p}( L_{\infty}(\mathbb{R}^d)\overline{\otimes}\mathcal{M})$ and $(y_{n})_{n}\subset  L_{\infty}(\mathbb{R}^d)\overline{\otimes}\mathcal{M}$ such that $f_{n}=a y_{n}b$  for all $n$, with 
	\begin{align*}
		\|a\|_{L_{2p}( L_{\infty}(\mathbb{R}^d)\overline{\otimes}\mathcal{M})}\sup_{n} \|y_{n}\|_{\infty}\,\, \|b \|_{L_{2p}( L_{\infty}(\mathbb{R}^d)\overline{\otimes}\mathcal{M})}\leq 	\|\sup_{n}f_{n} \|_{L_{p}( L_{\infty}(\mathbb{R}^d)\overline{\otimes}\mathcal{M})}+\delta.
	\end{align*}
	A direct computation shows that
	\begin{align*}
		 \|\sup_{n}f_{n} ( {\cdot}/{s}& ) \|_{L_{p}( L_{\infty}(\mathbb{R}^d)\overline{\otimes}\mathcal{M})}\\
		\leq&  \|a ( {\cdot}/{s})\|_{L_{2p}( L_{\infty}(\mathbb{R}^d)\overline{\otimes}\mathcal{M})}\sup_{n} \|y_{n} ({\cdot}/{s})\|_{\infty}\,\, \|b ( {\cdot}/{s})\|_{L_{2p}( L_{\infty}(\mathbb{R}^d)\overline{\otimes}\mathcal{M})} \\
		=&s^{\frac{d}{p}} \|a\|_{L_{2p}( L_{\infty}(\mathbb{R}^d)\overline{\otimes}\mathcal{M})}\sup_{n}\|y_{n}\|_{\infty}\,\|b\|_{L_{2p}( L_{\infty}(\mathbb{R}^d)\overline{\otimes}\mathcal{M})} \\
		\leq &s^{\frac{d}{p}}(\|\sup_{n}f_{n}\|_{L_{p}( L_{\infty}(\mathbb{R}^d)\overline{\otimes}\mathcal{M})}+\delta).
	\end{align*}
	Since $\delta$ is  arbitrary, we conclude that
	\begin{align*}
			\|\sup_{n}f_{n} ({\cdot}/{s})\|_{L_{p}( L_{\infty}(\mathbb{R}^d)\overline{\otimes}\mathcal{M})}\leq s^{\frac{d}{p}}\|\sup_{n}f_{n}\|_{L_{p}( L_{\infty}(\mathbb{R}^d)\overline{\otimes}\mathcal{M})}.
	\end{align*}
	The reverse inequality follows by the same argument.
\end{proof}

\noindent\textbf{2.3. Noncommutative square functions}.
	To introduce  the noncommutative square function,	we  first recall the   column and row   spaces. For a finite sequence $ (x_{n})_{n}$ in $L_{p}(\mathcal{M})$ with $1\leq p\leq\infty$,    define 
\begin{align*}
	\|(x_{n})_{n}\|_{L_{p}(\mathcal{M};\,\ell_{2}^{c})}=\bigg\|\bigg(\sum_{n}|x_{n}|^2\bigg)^{ {1}/{2}}\bigg\|_{L_{p}(\mathcal{M})},\quad  	\|(x_{n})_{n}\|_{L_{p}(\mathcal{M};\,\ell_{2}^{r})}=\bigg\|\bigg(\sum_{n}|x_{n}^{*}|^2\bigg)^{ {1}/{2}}\bigg\|_{L_{p}(\mathcal{M})}.
\end{align*}
Notice that if $p\neq2$, above two norms are not comparable. The column  space $L_{p}(\mathcal{M};\,\ell_{2}^{c})$ (resp. the row space $L_{p}(\mathcal{M};\,\ell_{2}^{r})$) is the   completion of all finite sequences in $L_{p}(\mathcal{M})$ with  respect to $\|\cdot\|_{L_{p}(\mathcal{M};\,\ell_{2}^{c})}$ (resp. $\|\cdot\|_{L_{p}(\mathcal{M};\,\ell_{2}^{r})}$).  
The noncommutative square function space $L_{p}(\mathcal{M};\,\ell_{2}^{cr})$ is defined as follows:
if $p\geq 2$,  $L_{p}(\mathcal{M};\,\ell_{2}^{cr})=L_{p}(\mathcal{M};\,\ell_{2}^{r})\cap L_{p}(\mathcal{M};\,\ell_{2}^{c})$ equipped with the intersection norm:
\begin{align*}
	\|(x_{n})_{n}\|_{L_{p}(\mathcal{M};\,\ell_{2}^{cr})}=\max\left\{\|(x_{n})_{n}\|_{L_{p}(\mathcal{M};\,\ell_{2}^{r})},\,\|(x_{n})_{n}\|_{L_{p}(\mathcal{M};\,\ell_{2}^{c})}   \right\};		
\end{align*}
if $1\leq p< 2$,  $L_{p}(\mathcal{M};\,\ell_{2}^{cr})=L_{p}(\mathcal{M};\,\ell_{2}^{r})+ L_{p}(\mathcal{M};\,\ell_{2}^{c})$ equipped with the sum norm:
\begin{align*}
	\|(x_{n})_{n}\|_{L_{p}(\mathcal{M};\,\ell_{2}^{cr})}=\inf\left\{\|(y_{n})_{n}\|_{L_{p}(\mathcal{M};\,\ell_{2}^{r})}+\|(z_{n})_{n}\|_{L_{p}(\mathcal{M};\,\ell_{2}^{c})}   \right\},		
\end{align*}
where the infimum runs over all decompositions $x_{n}=y_{n}+z_{n}$ with $y_n$ and $z_n$ in
$L_p(\mathcal{M})$.

By the definition of $L_{p}( \mathcal{M};\,\ell_{\infty})$ and $L_{p}(\mathcal{M};\,\ell_{2}^{cr})$, Hong et al.  get the following result (see \cite[Proposition 2.3]{hong2019pointwise}).
\begin{proposition}[\cite{hong2019pointwise}]\label{8.10.1}
	Let $1\leq p\leq \infty$ and $(x_{n})_{n}\in L_{p}( \mathcal{M};\,\ell_{\infty})$.   There exists an absolute
	constant $c>0$ such that  
	\begin{align*}
		\|\sup_{n} x_{n}\|_{L_{p}(\mathcal{M})}\leq c\|(x_n)_{n}\|_{L_{p}(\mathcal{M};\,l_{2}^{cr})}.
	\end{align*}	
\end{proposition}

\section{Two useful tools}\label{250414.1}
In this section, we  introduce  two useful tools.  The first is the dimension-free estimate for the semigroup $P_{t}$ on $\mathbb{Z}^d$, where $t>0$ and $P_{t}$  is a convolution operator with Fourier multiplier: 
\begin{align*} 
	\mathfrak{p}_{t}(\xi)=e^{-t\sum_{k=1}^{d}\sin^2(\pi \xi_{k})},\quad\xi\in \mathbb{T}^d.
\end{align*}
The second is the noncommutative analogue corresponding to the sampling principle in \cite[Corollary 2.1]{Magyar}, which  allows one to transfer $L_{p}(L_{\infty}(\mathbb{R}^d)\overline{\otimes}\mathcal{M};\,\ell_{\infty})$ norm of a sequence of convolution operators on $\mathbb{R}^d$ to $L_{p}(\mathcal{N};\,\ell_{\infty})$ norm  of their discrete analogues on $\mathbb{Z}^d$. Both of them are crucial components in our analysis, enabling us to establish Theorem \ref{0109.5}.

\subsection{dimension-free estimate for the semigroup $P_{t}$}
Let $ e_{1},\cdots ,e_{d} $ be the standard basis in $\mathbb{Z}^d$. For every $k\in\mathbb{N}_{d}$ and $x\in \mathbb{Z}^d$,  define 
\begin{align*}
	\Delta_{k} f(x)=f(x)-f(x+e_{k})
\end{align*}
as the discrete partial derivative on $\mathbb{Z}^d$. The adjoint of $\Delta_{k}$ is  determined by  $\Delta^{*}_{k} f(x)=f(x)-f(x-e_{k})$, and  the discrete partial Laplacian is given by
\begin{align}\label{250424.1} 
	\mathcal{L}_{k}f(x)={1}/{4}\cdot\Delta^{*}_{k}\Delta_{k}f(x)= {1}/{2}\cdot f(x)- {1}/{4}\big(f(x+e_{k})+f(x-e_{k})\big).
\end{align} 
We see that  
\begin{align*}
	\widehat{\mathcal{L}_{k}f}(\xi)=\frac{1-\cos(2\pi \xi_{k})}{2}\widehat{f}(\xi)=\sin^{2}(\pi\xi_{k})\widehat{f}(\xi),\quad\xi\in \mathbb{T}^d.
\end{align*}
\begin{proposition}\label{8.8.2}
	Let $1<p<\infty$ and $f\in L_{p}(\mathcal{N})$,  there exists a constant $C_{p}>0$ independent of  the dimension  $d$ such that  
	\begin{align*}
		\|\sup_{t>0} P_{t}f\|_{L_{p}(\mathcal{N})}\leq C_{p} \|f\|_{L_{p}(\mathcal{N})}.
	\end{align*}
\end{proposition}

  For this proof, we recall  the maximal inequality  in \cite[Corollary 5.11]{MR2276775}. 

\begin{lemma}\label{250806.1}
	Suppose that $(T_{t})_{t>0}: \mathcal{M}\rightarrow \mathcal{M}$ is a semigroup. For every $t>0$, the operator $T_{t}$ is   linear  and  satisfies:
	\begin{itemize}
		\item[(i)] $T_{t}$ is a contraction on $\mathcal{M}:$ $\|T_{t}(x)\|_{\infty}\leq\|x\|_{\infty}$ for all $x\in \mathcal{M}$;
		
		\vspace{0.1cm}
		
		\item[(ii)] $T_{t}$ is positive\,:~$ T_{t}(x)\geq0$  if $ x\geq0$;
		
		\vspace{0.1cm}
		
		\item [(iii)] $\tau \circ T_{t}\leq \tau:$ $\tau (T_{t}(x))\leq \tau(x)$ for all $x\in L_{1}(\mathcal{M})\cap \mathcal{M}_{+}$; 
		
		\vspace{0.1cm}
		
		\item[(iv)] $T_{t}$ is symmetric relative to $\tau:$ $\tau(T_{t}(y)^{*}x)=\tau(y^{*}T_{t}(x))$ for all $x,y\in L_{2}(\mathcal{M})\cap \mathcal{M}$.
	\end{itemize}
	Then, for $1<p<\infty$, we have
	\begin{align}\label{24127.1}
		\Big\|\sup_{t>0} T_{t}(x)\Big\|_{L_{p}(\mathcal{M})}\leq C_{p} 	\|x\|_{L_{p}(\mathcal{M})},\quad x \in L_{p}(\mathcal{M}).
	\end{align}
\end{lemma}

\noindent $\mathbf{Proof~of~Proposition~\ref{8.8.2}:}$ We first show that, for every $k\in \mathbb{N}_{d}$, the operator $\mathcal{L}_{k}$  given in (\ref{250424.1}) generates a semigroup $(e^{-t\mathcal{L}_{k}})_{t>0}$ satisfying (i)-(iv). Denoting  $\mathcal{G}_{k}f(x)=\frac{1}{2}\big(f(x+e_{k})+f(x-e_{k})\big)$,   one has $2\mathcal{L}_{k}=I-\mathcal{G}_{k}$. Hence,
\begin{align}\label{250515.1} 
	e^{-t\mathcal{L}_{k}}f=e^{-{t}/{2}}\sum_{n=0}^{\infty}\frac{({t}/{2})^{n}}{n!}(\mathcal{G}_{k})^{n}f,\quad t>0.
\end{align}
Formula (\ref{250515.1}), together with the positivity and symmetry of $\mathcal{G}_{k}$, 
implies that $e^{-t\mathcal{L}_{k}}$ is positive and symmetric. On the other hand, we calculate  
\begin{align*} 
	\|e^{-t\mathcal{L}_{k}}f\|_{\infty}
	\leq &e^{-{t}/{2}}\sum_{n=0}^{\infty}\frac{( {t}/{2})^{n}}{n!}\|(\mathcal{G}_{k})^{n}f\|_{\infty}\\
	\leq&e^{-{t}/{2}}\sum_{n=0}^{\infty}\frac{({t}/{2})^{n}}{n!}\|f\|_{\infty} 
	=\|f\|_{\infty}.
\end{align*}
Similarly, for any $f\in  L_{1}( \mathcal{N})\cap\mathcal{N}_{+}$, we conclude that 
\begin{align*}
	\sum \otimes\tau (e^{-t\mathcal{L}_{k}}   f)= \|e^{-t\mathcal{L}_{k}}f\|_{L_{1}(\mathcal{N})}
	\leq  \|f\|_{L_{1}(\mathcal{N})}= \sum \otimes\tau(f).
\end{align*}
In summary,  $(e^{-t\mathcal{L}_{k}})_{t>0}$ is a semigroup  satisfying (i)-(iv). Since the operators $\mathcal{L}_{1},\cdots,\mathcal{L}_{d}$ commute pairwise, we  obtain
\begin{align*} 
	e^{-t(\mathcal{L}_{1}+\cdots+\mathcal{L}_{d})}=e^{-t\mathcal{L}_{1}}\circ\cdots \circ e^{-t\mathcal{L}_{d}}, \quad t\geq 0.
\end{align*}
Consequently,  the operator $\mathcal{L}=\mathcal{L}_{1}+\cdots+\mathcal{L}_{d}$ generates a semigroup $(e^{-t\mathcal{L}})_{t>0}$ satisfying  (i)-(iv). By (\ref{24127.1}), one has
\begin{align}\label{250424.2}
	\|\sup_{t>0} e^{-t\mathcal{L}}f\|_{L_{p}(\mathcal{N})}\leq C_{p}\|f\|_{L_{p}(\mathcal{N})},\quad 1<p<\infty.
\end{align}
Notice  that
\begin{align}\label{250313.1}
	\widehat{e^{-t\mathcal{L}}f}(\xi)=e^{-t\sum_{k=1}^{d}\sin^2(\pi \xi_{k})}\widehat{f}(\xi)=\widehat{P_{t}f}(\xi).
\end{align}
Therefore,  combining (\ref{250424.2}) with (\ref{250313.1}), we  deduce that 
\begin{align*}
	\|\sup_{t>0} P_{t}f\|_{L_{p}(\mathcal{N})} \leq C_{p}\|f\|_{L_{p}(\mathcal{N})},\quad 1<p<\infty. 
\end{align*}
$\hfill\square$

Lemma \ref{1209.1} below is a direct application of Proposition \ref{8.8.2}.
\begin{lemma}\label{1209.1}
	Let $d\geq 2$ and $\mu$ be the normalized spherical measure given in  \eqref{250413.1}. For $t>0$,   define
	\begin{align*}
		a_{t}(\xi) =\widehat{\mu}\big(t(\xi-\llbracket \xi \rrbracket)\big),\quad \xi\in \mathbb{R}^d.
	\end{align*}
	Then, for all $f\in L_2(\mathcal{N})$,  there exists  a constant  $C>0$ independent of the dimension $d$ such that
	\begin{align*}
		\|\sup_{t\in \mathbb{D}} \mathcal{F}^{-1}(a_{t}\widehat{f}\,)\|_{L_{2}(\mathcal{N})}\leq C \|f\|_{L_{2}(\mathcal{N})}.
	\end{align*}
\end{lemma}
\begin{proof}
	By the triangle inequality, we have 
\begin{align}\label{250424.3}
		\|\sup_{t\in \mathbb{D}}\mathcal{F}^{-1}(a_{t}\widehat{f}\,)&\|_{L_{2}(\mathcal{N})}\nonumber\\
		\leq&\|\sup_{t\in\mathbb{D}}\mathcal{F}^{-1}(\mathfrak{p}_{t^2d^{-1}}  \widehat{f}\,)\|_{L_{2}(\mathcal{N})}+\|\sup_{t\in \mathbb{D}}\mathcal{F}^{-1}\big((a_{t}-\mathfrak{p}_{t^2d^{-1}})\widehat{f}\,\big)\|_{L_{2}(\mathcal{N})}.
\end{align}
Thanks to Proposition \ref{8.8.2}, there exists a constant $C>0$ independent of the dimension $d$ such that 
\begin{align*}
	\|\sup_{t\in\mathbb{D}}\mathcal{F}^{-1}(\mathfrak{p}_{t^2d^{-1}}  \widehat{f}\,)\|_{L_{2}(\mathcal{N})}=\|\sup_{t\in\mathbb{D}}P_{t^2d^{-1}}f \|_{L_{2}(\mathcal{N})}\leq C\|f\|_{L_{2}(\mathcal{N})}.
\end{align*}

	Now, we apply the noncommutative square functions to estimate the  second term on the right-hand side of (\ref{250424.3}). More precisely, based on  Proposition $\ref{8.10.1}$ and the Plancherel formula  (\ref{250609.1}), we see that
	\begin{align}\label{250508.1} 
		\|\sup_{t\in \mathbb{D}}\mathcal{F}^{-1}\big((a_{t}-\mathfrak{p}_{t^2d^{-1}})\widehat{f}\,\big)\|^2_{L_{2}(\mathcal{N})} 
		\lesssim& \bigg\|\bigg(\sum_{t\in \mathbb{D}}\big|\mathcal{F}^{-1}\big((a_{t}-\mathfrak{p}_{t^2d^{-1}})\widehat{f}\,\big)\big|^2\bigg)^{{1}/{2}}\bigg\|^2_{L_{2}(\mathcal{N})}\nonumber\\
		=&\sum_{t\in \mathbb{D}}\|\mathcal{F}^{-1}\big((a_{t}-\mathfrak{p}_{t^2d^{-1}})\widehat{f}\,\big)\|^2_{L_{2}(\mathcal{N})}\nonumber\\
		=&\sum_{t\in \mathbb{D}}\|(a_{t}-\mathfrak{p}_{t^2d^{-1}})\widehat{f}\,\|^2_{L_{2}(  L_{\infty}(\mathbb{T}^d)\overline{\otimes}\mathcal{M})}.
	\end{align}
	Note that	$a_{t}$ is $1$-periodic in each coordinate $\xi_{j}$ for $j\in \mathbb{N}_{d}$, so it is well defined on $\mathbb{T}^d$. For $\xi\in Q$, we have $\xi-\llbracket \xi \rrbracket=\xi$ and $|\xi-\llbracket \xi \rrbracket|=\|\xi\|$.  Recalling  \eqref{250424.5} and   the estimates of   $\widehat{\mu}$ shown in \cite[Lemma 4.2]{mirek2023dimension}:
	\begin{align*}
		&|\widehat{\mu}(\xi)-1|\leq 2\pi^2\big({|\xi|}/{\sqrt{d}}\big)^2,&\xi\in\mathbb{R}^d,\\
		&|\widehat{\mu}(\xi)|\lesssim \big({|\xi|}/{\sqrt{d}}\big)^{-{1}/{2}},&\xi\in\mathbb{R}^d,
	\end{align*}
	with the implicit constant independent of  $d$ and $\xi$, one has
	\begin{align}\label{1202.2}
		|a_{t}(\xi)-\mathfrak{p}_{t^2d^{-1}}(\xi)|\lesssim \min\left\{t^2d^{-1}\|\xi\|^2, t^{ -{1}/{2}}d^{ {1}/{4}}\|\xi\|^{-{1}/{2}} \right\},\quad\xi\in\mathbb{T}^d.
	\end{align}
Consequently, taking $t=2^{m}$ with $m\in\mathbb{N}$, and using $(\ref{1202.2})$, we conclude that 
\begin{align}\label{250515.5}
	\mbox{RHS}(\ref{250508.1})\lesssim& \tau\int_{\mathbb{T}^{d}}\sum_{m\in\mathbb{N}} \min\left\{\frac{2^{4m}\|\xi\|^4}{d^2},\,\,\frac{d^{{1}/{2}}}{2^{m} \|\xi\|}
	 \right\} |\widehat{f}(\xi)|^2d\xi\nonumber\\
	 \leq& \tau\int_{\mathbb{T}^{d}}\sum_{m\in\mathbb{N}\atop  m\geq \log_{2}{\sqrt{d}}/{ \|\xi\|}}\frac{d^{{1}/{2}}}{2^{m} \|\xi\|}\,|\widehat{f}(\xi)|^2d\xi
	+\tau\int_{\mathbb{T}^{d}}\sum_{m\in\mathbb{N}\atop  m<\log_{2}{\sqrt{d}}/{ \|\xi\|}}\frac{2^{4m}\|\xi\|^4}{d^2}\,|\widehat{f}(\xi)|^2d\xi\nonumber\\
	 \leq &2 \|f\|^2_{L_{2}(\mathcal{N})},
\end{align}
where the implicit constant is independent of the dimension $d$. This proof is completed.
\end{proof}

\begin{lemma}\label{1209.2}
	Let $d\geq 2$ and $f\in L_{2}(L_{\infty}(\mathbb{R}^d)\overline{\otimes}\mathcal{M})$, we have 
	\begin{align*}
		\|\sup_{t\in \mathbb{D}} \mathcal{F}^{-1}\big(\widehat{\mu}(t\cdot)\widehat{f}\,\big)\|_{L_{2}( L_{\infty}(\mathbb{R}^d)\overline{\otimes}\mathcal{M})}\lesssim \|f\|_{L_{2}( L_{\infty}(\mathbb{R}^d)\overline{\otimes}\mathcal{M})},
	\end{align*}
	where the implicit constant is independent of the  dimension $d$.
\end{lemma}
\begin{proof}
	This proof is similar to the proof of Lemma \ref{1209.1}, the only difference being the use of heat semigroup on $\mathbb{R}^d$ in place of $P_{t}$; we omit the details here.
\end{proof}

\subsection{Noncommutative sampling principle}\label{250413.2}
	Let $\Lambda$ be an index set. For every $\lambda\in \Lambda$, $T_{\lambda}$ is a convolution operator with a suitable kernel $K_{\lambda}$, i.e.,
	\begin{align*}
		T_{\lambda}f(x)=\int_{\mathbb{R}^d}K_{\lambda}(y)f(x-y)dy,\quad x\in \mathbb{R}^d.
	\end{align*}
	 Additionally,  we assume that 
	 \begin{align}\label{250527.3}
	 	 m_{\lambda}(\xi)=\widehat{K_{\lambda}}(\xi)\in L_{\infty}(\mathbb{R}^d),\quad \mbox{and}\quad \mbox{supp}\, m_{\lambda}\subset 	Q.
	 \end{align}
	  Define $(T_{\lambda})_{dis}$ as the convolution operator  on $\mathbb{Z}^d$ with  kernel   $(K_{\lambda})_{dis}=K_{\lambda}|_{\mathbb{Z}^d}$, more precisely,
\begin{align*} 
 (T_{\lambda})_{dis}f(n)=\sum_{m\in\mathbb{Z}^d}K_{\lambda}(m)f(n-m),\quad  n\in\mathbb{Z}^d.
\end{align*}
\begin{lemma}[\cite{chen-hong}]\label{250428.3}
	Let   $(T_{\lambda})_{\lambda\in\Lambda}$ be a family of convolution operators satisfying assumption \eqref{250527.3}.	Suppose that for some $1<p<\infty$,  there exists a  constant $C>0$ independent of the dimension $d$ such that  
	\begin{align*}
			\|\sup_{\lambda\in\Lambda} T_{\lambda} g\|_{L_{p}( L_{\infty}(\mathbb{R}^d)\overline{\otimes}\mathcal{M})}\leq C\|g\|_{L_{p}( L_{\infty}(\mathbb{R}^d)\overline{\otimes}\mathcal{M})},\quad g\in L_{p}( L_{\infty}(\mathbb{R}^d)\overline{\otimes}\mathcal{M}).		
	\end{align*}
	Then, for every $f\in L_{p}(\mathcal{N})$, we have
	\begin{align*}
		\|\sup_{\lambda\in\Lambda}\,(T_{\lambda})_{dis}f\|_{L_{p}(\mathcal{N})}\lesssim 3^d\|f\|_{L_{p}(\mathcal{N})},
	\end{align*}
	where 	  the implicit constant is independent of  the dimension $d$.
\end{lemma}

Fix $q \in\mathbb{N}$ and make the stronger assumption:
\begin{align}\label{250604.1}
	  \mbox{Im}(K_{\lambda})=0, \quad  \mbox{and}\quad \mbox{supp}\, m_{\lambda}\subset 	q^{-1}Q.
\end{align}
Here, $\mbox{Im}(K_{\lambda})$ represents the imaginary part of $K_{\lambda}$. Assumption (\ref{250604.1}) implies that $K_{\lambda}$ is a real-valued function. Define
\begin{align}\label{250603.2}
	(m_{\lambda})_{per}^{q}(\xi)=\sum_{x\in\mathbb{Z}^d}m_{\lambda}(\xi-{x}/{q}),\quad\xi\in \mathbb{R}^d,
\end{align}
which is ${1}/{q}$ periodic in each coordinate. We consider the associated multiplier operator  $(T_{\lambda})_{dis}^{q}$,   given by 
\begin{align}\label{250516.1}
	\widehat{(T_{\lambda})_{dis}^{q}f}(\xi)=(m_{\lambda})_{per}^{q}(\xi)\widehat{f}(\xi).
\end{align}

\begin{proposition}\label{1206.1}
			Let   $(T_{\lambda})_{\lambda\in\Lambda}$ be a family of convolution operators satisfying assumption  \eqref{250604.1}.  Suppose that for some $1<p<\infty$,  there exists a  constant $C>0$ independent of the dimension $d$ such that 
	\begin{align}\label{250427.1}
		\|\sup_{\lambda\in\Lambda}T_{\lambda}g\|_{L_{p}( L_{\infty}(\mathbb{R}^d)\overline{\otimes}\mathcal{M})}\leq C\|g\|_{L_{p}( L_{\infty}(\mathbb{R}^d)\overline{\otimes}\mathcal{M})}, \quad g\in L_{p}( L_{\infty}(\mathbb{R}^d)\overline{\otimes}\mathcal{M}).
	\end{align}
	Then, for every $f\in L_{p}(\mathcal{{N}})$, we have
	\begin{align*}
		\|\sup_{\lambda\in\Lambda}\,(T_{\lambda})^{q}_{dis}f\|_{L_{p}(\mathcal{{N}})}\lesssim 3^d\|f\|_{L_{p}(\mathcal{{N}})},
	\end{align*}
	where the implicit constant is independent of  the dimension $d$.
\end{proposition}
\begin{proof}
Let $T_{\lambda}^{q}$ be  the convolution operator  with  kernel $(K_{\lambda})_{{1}/{q}}(\cdot)=q^{d}K_{\lambda}(q\,\cdot)$.  A scaling argument combined with  Lemma \ref{1203.4}  shows that, for all $s>0$,
	\begin{align}\label{1206.2}
		\|(T^{s}_{\lambda})_{\lambda\in\Lambda}\|_{L_{p}( L_{\infty}(\mathbb{R}^d)\overline{\otimes}\mathcal{M})\rightarrow L_{p}( L_{\infty}(\mathbb{R}^d)\overline{\otimes}\mathcal{M};\,\ell_{\infty})}=\|(T_{\lambda})_{\lambda\in\Lambda}\|_{L_{p}(L_{\infty}(\mathbb{R}^d)\overline{\otimes}\mathcal{M})\rightarrow L_{p}( L_{\infty}(\mathbb{R}^d)\overline{\otimes}\mathcal{M};\,\ell_{\infty})}.
	\end{align}
	Indeed,	 
	\begin{align*}
			\mbox{LHS}(\ref{1206.2})=&\sup_{g\neq 0}\frac{\|\sup_{\lambda\in\Lambda}T^{s}_{\lambda}g\|_{L_{p}( L_{\infty}(\mathbb{R}^d)\overline{\otimes}\mathcal{M})}}{\|g\|_{L_{p}( L_{\infty}(\mathbb{R}^d)\overline{\otimes}\mathcal{M})}}\nonumber\\
			=&\sup_{g\neq 0}\frac{\|\sup_{\lambda\in\Lambda}T_{\lambda}g_{s}(s\cdot)\|_{L_{p}( L_{\infty}(\mathbb{R}^d)\overline{\otimes}\mathcal{M})}\,s^{d}}{\|g\|_{L_{p}( L_{\infty}(\mathbb{R}^d)\overline{\otimes}\mathcal{M})}}\nonumber \\=&\sup_{g\neq 0}\frac{\|\sup_{\lambda\in\Lambda}T_{\lambda}g_{s}\|_{L_{p}( L_{\infty}(\mathbb{R}^d)\overline{\otimes}\mathcal{M})}\,s^{-\frac{d}{p}}s^{d}}{\|g\|_{L_{p}( L_{\infty}(\mathbb{R}^d)\overline{\otimes}\mathcal{M})}}\nonumber\\
		=&\sup_{g\neq 0}\frac{\|\sup_{\lambda\in\Lambda}T_{\lambda}g_{s}\|_{L_{p}( L_{\infty}(\mathbb{R}^d)\overline{\otimes}\mathcal{M})}\,s^{-\frac{d}{p}}s^{d}}{\|g_{s}\|_{L_{p}( L_{\infty}(\mathbb{R}^d)\overline{\otimes}\mathcal{M})}\,s^{-\frac{d}{p}}s^{d}}=\mbox{RHS}(\ref{1206.2}).
	\end{align*}
Then, combining Lemma \ref{250428.3}, assumption (\ref{250427.1}), and identity (\ref{1206.2}),  we obtain
	\begin{align}\label{0116.1}
		\|\sup_{\lambda\in\Lambda}\,(T_{\lambda}^{q})_{dis}f\|_{L_{p}(\mathcal{{N}})}\lesssim 3^d\|f\|_{L_{p}(\mathcal{{N}})},\quad f\in L_{p}(\mathcal{{N}}),
	\end{align}
where the implicit constant is independent of the dimension $d$.

	Now, we emphasize that $(T_{\lambda}^{q})_{dis}\neq (T_{\lambda})^{q}_{dis}$. 	Denote by $K_{\lambda}^{\#}$  the convolution kernel of $(T_{\lambda})^{q}_{dis}$.  It was shown in  \cite[Page 196]{Magyar} that
	$$
	K_{\lambda}^{\#}(n)=
	\begin{cases}
		q^dK_{\lambda}(n), & n\in q\mathbb{Z}^d,\\
		0,& n\in\mathbb{Z}^d, \quad\mbox{but} ~n\notin q\mathbb{Z}^d,
	\end{cases}
	$$
	which does not coincide with the convolution kernel of $(T_{\lambda}^{q})_{dis}$, given by  ${(K_{\lambda})}_{{1}/{q}}(n)=q^{d}K_{\lambda}(qn), \, n\in\mathbb{Z}^d$.
 
	Let $T_{\lambda}^{\#}$  be the convolution  operator with  kernel $K_{\lambda}^{\#}$, which maps functions on  $q\mathbb{Z}^d$ to $\mathcal{{M}}$, i.e.,
	\begin{align}\label{2485.13}
		T_{\lambda}^{\#}f(nq)=\sum_{m\in\mathbb{Z}^d}f(nq-mq)K_{\lambda}^{\#}(mq)=\sum_{m\in\mathbb{Z}^d}f(nq-mq)q^dK_{\lambda}(mq).
	\end{align}
Furthermore, define a function $\varrho(f)$  mapping from $\mathbb{Z}^d$ to $\mathcal{ {M}}$  by 
\begin{align*}
	\varrho(f)(n)=f(nq).
\end{align*}
It is clear that 
\begin{align}\label{250428.1}
		(T_{\lambda}^{q})_{dis}\varrho(f)(n)=&\sum_{m\in\mathbb{Z}^d}\varrho(f)(n-m)(K_{\lambda})_{{1}/{q}}(m)\nonumber
		\\=& \sum_{m\in\mathbb{Z}^d}\varrho(f)(n-m)q^dK_{\lambda}(mq)\nonumber\\ 
	=& \sum_{m\in\mathbb{Z}^d}f(nq-mq)q^dK_{\lambda}(mq)=T_{\lambda}^{\#}f(nq).
\end{align}
By invoking (\ref{0116.1})  and (\ref{250428.1}),  we conclude that   
\begin{align}\label{1206.3}
	\|\sup_{\lambda\in\Lambda}T^{\#}_{\lambda}f\|_{L_{p}(\ell_{\infty}(q\mathbb{Z}^d)\overline{\otimes}\mathcal{M})}=&\|\sup_{\lambda\in\Lambda}\,(T^{q}_{\lambda})_{dis}\varrho(f)\|_{L_{p}(\mathcal{{N}})}\nonumber\\
	\lesssim&3^d\|\varrho(f)\|_{L_{p}(\mathcal{{N}})}=3^d\|f\|_{L_{p}(\ell_{\infty}(q\mathbb{Z}^d)\overline{\otimes}\mathcal{M})},
\end{align} 
where the last equality follows from 
\begin{align*} 
	\|	\varrho(f)\|^{p}_{L_{p}(\mathcal{{N}})}=\tau \sum_{n\in\mathbb{Z}^d}|	\varrho(f)(n)|^{p} 
	=\tau \sum_{n\in\mathbb{Z}^d}|f(nq)|^{p} =\|f\|^{p}_{L_{p}(\ell_{\infty}(q\mathbb{Z}^d)\overline{\otimes}\mathcal{M})}.
\end{align*}

  For every $n\in\mathbb{Z}^d$, there exist  unique elements $n_{0}\in\mathbb{Z}^d$ and  $n_{1}\in\mathbb{Z}^d/q\mathbb{Z}^d$ such that 
 \begin{align*}
 	n=qn_{0}+n_{1}.
 \end{align*}
Given a function $f:\mathbb{Z}^d\rightarrow \mathcal{{M}}$,  we define the corresponding function  $f_{n_{1}}:q\mathbb{Z}^d\rightarrow \mathcal{{M}}$  by 
\begin{align*}
 f_{n_{1}}(qn_{0})=f(qn_{0}+n_{1}).
\end{align*}
we assert   that 
	\begin{align}\label{250515.4}
		(T_{\lambda})_{dis}^{q}f(n)=T_{\lambda}^{\#}f_{n_{1}}(qn_{0}).
	\end{align}
Indeed,  applying (\ref{2485.13}), we obtain 
	\begin{align*} 
		T_{\lambda}^{\#}f_{n_{1}}(qn_{0})
		=&\sum_{m\in\mathbb{Z}^d}f_{n_{1}}(qn_{0}-qm)q^dK_{\lambda}(qm)\nonumber\\
		=& \sum_{m\in\mathbb{Z}^d}f(n-qm)q^dK_{\lambda}(qm),
	\end{align*}
which coincides with 
	\begin{align*}
	(T_{\lambda})^{q}_{dis}f(n)=\sum_{m\in\mathbb{Z}^d}f(n-m)K_{\lambda}^{\#}(m)=\sum_{m\in\mathbb{Z}^d}f(n-qm)q^dK_{\lambda}(qm).
\end{align*}

	 From now on, without loss of
	generality, we assume that  $f\geq 0$. Then $ T^{\#}_{\lambda}f_{n_{1}}(\cdot)$ is a selfadjoint operator  since $ K_{\lambda}$ is a real-valued function.   Combining   $(\ref{1206.3})$ with \eqref{250427.2}, there exists  a positive operator-valued function $G_{n_{1}}(\cdot)\in L_{p}(\ell_{\infty}(q\mathbb{Z}^d)\overline{\otimes}\mathcal{M})_{+}$ such that 
	\begin{align}\label{250515.2}
		-G_{n_{1}}(\cdot)\leq T^{\#}_{\lambda}f_{n_{1}}(\cdot) \leq  G_{n_{1}}(\cdot),\quad\forall\,\lambda\in\Lambda,
	\end{align}
	with
	\begin{align}\label{250515.3}
		\|G_{n_{1}}(\cdot)\|_{L_{p}(\ell_{\infty}(q\mathbb{Z}^d)\overline{\otimes}\mathcal{M})}\lesssim 3^d \,\|f_{n_{1}}(\cdot)\|_{L_{p}(\ell_{\infty}(q\mathbb{Z}^d)\overline{\otimes}\mathcal{M})}.
	\end{align}
	Consequently, by (\ref{250515.4})-(\ref{250515.3}), we obtain 
	\begin{align*}
		\|\sup_{\lambda\in\Lambda}\,(T_{\lambda})^{q}_{dis}f\|_{L_{p}(\mathcal{{N}})}=&\|\sup_{\lambda\in\Lambda}\,T_{\lambda}^{\#}f_{n_{1}}(\cdot)\|_{L_{p}(\mathcal{{N}})}\\
		\leq & \big\|\|G_{n_{1}}(\cdot)\|^p_{L_{p}(\ell_{\infty}(q\mathbb{Z}^d)\overline{\otimes}\mathcal{M})}\big\|_{\ell_{1}(\mathbb{Z}^d/q\mathbb{Z}^d)}^\frac{1}{p}\\
		\lesssim& 3^d \,\big\|\|f_{n_{1}}(\cdot)\|^p_{L_{p}(\ell_{\infty}(q\mathbb{Z}^d)\overline{\otimes}\mathcal{M})}\big\|_{\ell_{1}(\mathbb{Z}^d/q\mathbb{Z}^d)}^\frac{1}{p} 
		=  3^d\,\|f\|_{L_{p}(\mathcal{{N}})},
	\end{align*}
	which completes the proof.
\end{proof}

Further,  we consider another convolution operator, whose Fourier multiplier is akin to (\ref{250603.2}). It takes the form
\begin{align*}
	m(\xi)=\sum_{x\in\mathbb{Z}^d}\gamma_{x}\Psi(\xi- {x}/{q}),\quad\xi\in\mathbb{R}^d,
\end{align*}
satisfying 
\begin{itemize}
	\item[(i)]  $\Psi\in C^{\infty}_{c}(q^{-1}Q)$ and $\sum_{n\in\mathbb{Z}^d}\widehat{\Psi}(n)\leq A$ for some positive constant $A$.
	
		\vspace{0.1cm}
	
	\item[(ii)]  $(\gamma_{x})_{x\in\mathbb{Z}^d}$ is a $q\mathbb{Z}^d$ periodic sequence, that is, $\gamma_{x}=\gamma_{x{'}}$, whenever  $x-x{'}\in q\mathbb{Z}^d$.
\end{itemize}

It was shown in \cite[Lemma 5.3]{chen-hong} that  
\begin{align}\label{250411.1}
	\|\mathcal{F}^{-1}(m\widehat{f}\,)\|_{L_{2}(\mathcal{{N}})}\leq A\sup_{x\in\mathbb{Z}^d}|\gamma_{x}|\|f\|_{L_{2}(\mathcal{{N}})}.
\end{align}

\section{Proof of Theorem \ref{0109.5}}\label{0118.4}
In this section, we prove the dimension-free estimate of the noncommutative discrete spherical maximal operator on $L_{p}(\mathcal{N})$ for $ 2\leq p\leq\infty$.  If $p=\infty$, there is nothing to do, since  $\mathcal{A}^{d}_{t}$ is an averaging operator. By
Theorem \ref{0.4}, it suffices to show 
	\begin{align}\label{250610.1} 
	\|\sup_{t\in \mathbb{D}}\mathcal{A}_{t}^{d}f\|_{L_{2}(\mathcal{N})}\lesssim\|f\|_{L_{2}(\mathcal{N})},
\end{align}
where  the implicit  constant is independent of the dimension $d$. In the following, we   consider the maximal functions corresponding to the operators $\mathcal{A}_{t}^{d}$ in which the supremum is restricted respectively to the sets:
\begin{itemize}
	\item[(i)]  the small-scale case:
	\begin{align*}
		\mathcal{D}_{c_{0}}=\left\{t \in\mathbb{D}:1\leq t\leq c_{0}d^{{1}/{2}}\right\};
	\end{align*}
	 	\item[(ii)]   the intermediate-scale case: 
	 	\begin{align*}
	 		\mathcal{D}_{c_{1},c_{2}}=\left\{t \in\mathbb{D}:c_{1}d^{{1}/{2}}\leq t\leq c_{2}d^{{3}/{2}}\right\};
	 	\end{align*}
	 	 \item[(iii)]   the large-scale case:
	 	 \begin{align*}
	 	 	\mathcal{D}_{c_{3},\infty}=\left\{t \in\mathbb{D}:t\geq c_{3}d^{{3}/{2}}\right\},
	 	 \end{align*}
\end{itemize}
for some universal constants $c_{0}, c_{1}, c_{2}, c_{3}>0$. Since we are working with the dyadic numbers $\mathbb{D}$, the
values of $c_{0}, c_{1}, c_{2}, c_{3}$  never play a role as long as they are absolute constants. Moreover, the implied constant in (\ref{250610.1}) is allowed to depend on  $c_{0}, c_{1}, c_{2}, c_{3}$.

\subsection{The small-scale and intermediate-scale cases}
This subsection is devoted to estimating the
maximal functions corresponding to $\mathcal{A}_{t}^{d}$ with the supremum taken over the sets $\mathcal{D}_{c_{0}}$ and $\mathcal{D}_{c_{1},c_{2}}$, respectively. 

\begin{theorem}\label{0108.4}
	Let $c_{0}> 0$ and  define $\mathcal{D}_{c_{0}} =\left\{t \in\mathbb{D}:1\leq t\leq c_{0}d^{ {1}/{2}}\right\}$. For every $f\in L_{2}(\mathcal{N})$, there exists a constant $C>0$ independent of the dimension $d$ such that
	\begin{align*} 
		\|\sup_{t\in \mathcal{D}_{c_{0}}}\mathcal{A}_{t}^{d}f\|_{L_{2}(\mathcal{N})}\leq C\|f\|_{L_{2}(\mathcal{N})}.
	\end{align*}
\end{theorem}

For convenience, we write the operator $\mathcal{A}_{t}^{d}$  in convolution form with   kernel
\begin{align*}
	\mathcal{K}_{t}(x)=\frac{1}{|\mathbb{S}_{t}\cap\mathbb{Z}^d|}\chi_{\mathbb{S}_{t}\cap\mathbb{Z}^d}(x),\quad x\in\mathbb{Z}^d.
\end{align*}
The corresponding multipliers are given by
\begin{align}\label{250423.1}
	\mathfrak{m}_{t}(\xi) =\widehat{\mathcal{K}_{t}}(\xi)=\frac{1}{|\mathbb{S}_{t}\cap\mathbb{Z}^d|}\sum_{x\in\mathbb{S}_{t}\cap\mathbb{Z}^d}e^{2\pi \mathrm{i}\langle \xi,x\rangle}.
\end{align}
Let
\begin{align*} 
	V_{\xi}=\left\{j\in \mathbb{N}_{d}: \cos(2\pi \xi_{j})<0\right\}=\left\{j\in \mathbb{N}_{d}:  {1}/{4}<\|\xi_{j}\|\leq {1}/{2}\right\},\quad \xi\in\mathbb{T}^d. 
\end{align*} 
The following  two  suitable multipliers  are used to prove Theorem \ref{0108.4},
\begin{align*}
	&p_{\lambda}^{1}(\xi) =e^{-\kappa(d,\lambda)^2\sum_{j=1}^{d}\sin^2(\pi \xi_{j})}, &\mbox{if } |V_{\xi}|\leq{d}/{2},\\
	&p_{\lambda}^{2}(\xi) =(-1)^{\lambda}e^{-\kappa(d,\lambda)^2\sum_{j=1}^{d}\cos^2(\pi \xi_{j})}, &\mbox{if } |V_{\xi}|>{d}/{2}.
\end{align*} 
Both  of them are expressed in terms of a proportionality constant:
\begin{align*}
  \kappa(d,\lambda)=\bigg(\frac{\lambda}{d}\bigg)^\frac{1}{2}=\frac{t}{\sqrt{d}},
\end{align*}
where  $\lambda=t^2$ throughout the paper.

\begin{proposition}[\cite{mirek2023dimension}]\label{1205.4}
	Let $d\geq 5$ and suppose that $\kappa(d,\lambda)\leq  {1}/{5}$. For every $\xi\in\mathbb{T}^d$ and $\lambda\in\mathbb{N}$, there exists a constant $0<c<1$ such that 
	\begin{itemize}
			\item[(i)]  if $|V_{\xi}|\leq {d}/{2}$, then 
		\begin{align*} 
			|\mathfrak{m}_{t}(\xi)-p_{\lambda}^{1}(\xi)|\lesssim\min\left\{e^{-\frac{c\kappa(d,\lambda)^2}{400}\sum_{j=1}^{d}\sin^2(\pi \xi_{j})},\kappa(d,\lambda)^2\sum_{j=1}^{d}\sin^2(\pi \xi_{j})\right\};
		\end{align*}
			\item[(ii)]  if $|V_{\xi}|>\frac{d}{2}$, then 
		\begin{align*} 
			|\mathfrak{m}_{t}(\xi)-p_{\lambda}^{2}(\xi)|\lesssim\min\left\{e^{-\frac{c\kappa(d,\lambda)^2}{400}\sum_{j=1}^{d}\cos^2(\pi \xi_{j})},\kappa(d,\lambda)^2\sum_{j=1}^{d}\cos^2(\pi \xi_{j})\right\}.
		\end{align*}
	\end{itemize}
\end{proposition}

\noindent $\mathbf{Proof~of~Theorem~\ref{0108.4}:}$ Let
$f\in L_{2}(\mathcal{N})$ and decompose it as $f=f_{1}+f_{2}$, where $\widehat{f_{1}}(\xi)=\widehat{f}(\xi)\chi_{\left\{\eta\in \mathbb{T}^{d}:~|V_{\eta}|\leq  {d}/{2}\right\}}(\xi)$. Applying the triangle inequality, we obtain
\begin{align}\label{2485.5}
	\|\sup_{t\in \mathcal{D}_{c_{0}} }\mathcal{A}_{t}^df&\|_{L_{2}\mathcal{(N)}}=\|\sup_{t\in \mathcal{D}_{c_{0}}}\mathcal{F}^{-1}(\mathfrak{m}_{t}\widehat{f}\,)\|_{L_{2}\mathcal{(N)}}\nonumber\\
	&\leq\sum_{k=1}^{2}\|\sup_{t\in \mathcal{D}_{c_{0}}}\mathcal{F}^{-1}(p_{t^2}^{k}\widehat{f_{k}})\|_{L_{2}\mathcal{(N)}}+\sum_{k=1}^{2}\|\sup_{t\in \mathcal{D}_{c_{0}}}\mathcal{F}^{-1}\big((\mathfrak{m}_{t}-p_{t^2}^{k})\widehat{f_{k}}\big)\|_{L_{2}\mathcal{(N)}}.
\end{align}
Recalling  $p_{t^2}^{1}(\xi)=\mathfrak{p}_{\kappa(d,\lambda)^2}(\xi)$ and Proposition \ref{8.8.2},  one has
\begin{align}\label{1220.6}
	\|\sup_{t\in \mathcal{D}_{c_{0}}}\mathcal{F}^{-1}(p_{t^2}^{1}\widehat{f_{1}})\|_{L_{2}(\mathcal{N})}=\|\sup_{t\in \mathcal{D}_{c_{0}}}P_{\kappa(d,\lambda)^2}f_1\|_{L_{2}(\mathcal{N})}\lesssim\|f_{1}\|_{L_{2}(\mathcal{N})}\leq \|f\|_{L_{2}(\mathcal{N})},
\end{align}
where the implicit constant is independent of the dimension $d$. We claim that 
\begin{align*} 
	\|\sup_{t\in \mathcal{D}_{c_{0}}}\mathcal{F}^{-1}(p_{t^2}^{2}\widehat{f_{2}})\|_{L_{2}(\mathcal{N})}\lesssim\|f\|_{L_{2}(\mathcal{N})} 
\end{align*}
is deduced from (\ref{1220.6}). In fact, let $F_{2}(x)=(-1)^{\sum_{j=1}^{d} x_{j}}f_{2}(x)$, and   calculate  
\begin{align*}
	\widehat{F_{2}}(\xi)&=\sum_{x\in \mathbb{Z}^{d}}e^{-2\pi \mathrm{i} \langle x, \xi \rangle }(-1)^{\sum_{j=1}^{d} x_{j}}f_{2}(x)\\
	&=\sum_{x\in \mathbb{Z}^{d}}e^{-2\pi \mathrm{i} \langle x, \xi \rangle}(e^{\pi \mathrm{i}})^{\sum_{j=1}^{d} x_{j}}f_{2}(x)\\
	&=\sum_{x\in \mathbb{Z}^{d}}e^{-2\pi \mathrm{i} \langle x, \xi-{ {\textbf{1}}/{2}}\rangle}f_{2}(x)=\widehat{f_{2}}(\xi-{{\textbf{1}}/{2}}),
\end{align*}
where $\mathbf1=(1,\cdots,1)\in\mathbb{Z}^d$. Consequently,
\begin{align*}
	\|\sup_{t\in \mathcal{D}_{c_{0}}} \mathcal{F}^{-1}(p_{t^2}^{2}\widehat{f_{2}})\|_{L_{2}(\mathcal{N})}
	=&	\bigg\|\sup_{t\in \mathcal{D}_{c_{0}}} \int_{\mathbb{T}^{d}}p_{t^2}^{2}(\xi)\widehat{f_{2}}(\xi)e^{2\pi \mathrm{i} \langle \cdot, \xi \rangle}d\xi\bigg\|_{L_{2}(\mathcal{N})} \\
	=&\bigg\|\sup_{t\in \mathcal{D}_{c_{0}}} \int_{\mathbb{T}^{d}}(-1)^{\lambda}e^{-\kappa(d,\lambda)^2\sum_{j=1}^{d}\cos^2(\pi \xi_{j})}\widehat{f_{2}}(\xi)e^{2\pi \mathrm{i} \langle \cdot, \xi \rangle}d\xi\bigg\|_{L_{2}(\mathcal{N})}\\
	=&\|\sup_{t\in \mathcal{D}_{c_{0}}} \mathcal{F}^{-1}(p_{t^2}^{1}\widehat{F_{2}})(\cdot)\, e^{ -\pi \mathrm{i} \langle \cdot,\textbf{1}  \rangle}\|_{L_{2}(\mathcal{N})}\\ 
	\leq&\|\sup_{t\in \mathcal{D}_{c_{0}}} \mathcal{F}^{-1}(p_{t^2}^{1}\widehat{F_{2}})\|_{L_{2}(\mathcal{N})} 
	\lesssim\|f\|_{L_{2}(\mathcal{N})}, 
\end{align*}
where we  use Lemma \ref{0109.1}, (\ref{1220.6}), along with the fact that  $\|F_{2}\|_{L_{2}(\mathcal{N})}=\|f_{2}\|_{L_{2}(\mathcal{N})}\leq \|f\|_{L_{2}(\mathcal{N})}$.

Now, we apply Proposition \ref{8.10.1} to estimate the  second term on the  right-hand side of (\ref{2485.5}), beginning with the case $k=1$. More precisely,    
\begin{align*} 
	\|\sup_{t\in \mathcal{D}_{c_{0}}}\mathcal{F}^{-1}\big((\mathfrak{m}_{t}-p_{t^2}^{1})\widehat{f_{1}}\big)\|^2_{L_{2}(\mathcal{N})}
	\lesssim&\bigg\|\bigg(\sum_{t\in \mathcal{D}_{c_{0}}}\big|\mathcal{F}^{-1}\big((\mathfrak{m}_{t}-p_{t^2}^{1})\widehat{f_{1}}\big)\big|^2\bigg)^{ {1}/{2}}\bigg\|^2_{L_{2}(\mathcal{N})}\nonumber\\
	=& \sum_{t\in \mathcal{D}_{c_{0}}}\|\mathcal{F}^{-1}\big((\mathfrak{m}_{t}-p_{t^2}^{1})\widehat{f_{1}}\big)\|^2_{L_{2}(\mathcal{N})}\nonumber\\
	=&\sum_{t\in \mathcal{D}_{c_{0}}}\tau\int_{\mathbb{T}^{d}}|\mathfrak{m}_{t}(\xi)-p_{t^2}^{1}(\xi)|^{2}|\widehat{f_{1}}(\xi)|^2d\xi.
\end{align*}
Similar to (\ref{250515.5}) in the proof of Lemma \ref{1209.1}, by setting $t=2^{m}\in \mathcal{D}_{c_{0}} $ with $m\in\mathbb{N}$, and applying Proposition \ref{1205.4},
we obtain 
\begin{align*}
	\|\sup_{t\in \mathcal{D}_{c_{0}}}\mathcal{F}^{-1}\big((\mathfrak{m}_{t}-p_{t^2}^{1})\widehat{f_{1}}\big)\|_{L_{2}\mathcal{(N)}}\lesssim \|f_{1}\|_{L_{2}\mathcal{(N)}}\leq \|f\|_{L_{2}\mathcal{(N)}},
\end{align*}
where the implicit constant is independent of the dimension $d$. 

The same argument also applies to $k=2$, which completes the proof.  $ \hfill\square $

\begin{theorem}\label{0108.5}
	Let $c_{1},c_{2}> 0$ and define  $\mathcal{D}_{c_{1},c_{2}}=\left\{t \in\mathbb{D}:c_{1}d^{{1}/{2}}\leq t\leq c_{2}d^{{3}/{2}}\right\}$. Then for every $f\in L_{2}\mathcal{(N)}$, there exists a constant $C> 0$ independent of the dimension $d$ such that  
	\begin{align*} 
		\|\sup_{t\in \mathcal{D}_{c_{1},c_{2}}}\mathcal{A}_{t}^{d}f\|_{L_{2}(\mathcal{N})}\leq C\|f\|_{L_{2}(\mathcal{N})}.
	\end{align*}
\end{theorem}

The strategy of proving  Theorem \ref{0108.5}  is similar to the proof of Theorem \ref{0108.4}. We omit the details here. The only difference is the use of following Proposition \ref{1205.5} in place of Proposition \ref{1205.4}. 

\begin{proposition}[\cite{mirek2023dimension}]\label{1205.5}
	Let $d, \lambda \in \mathbb{N}$ with $100d\leq\lambda\leq d^3$. For every $\xi\in\mathbb{T}^d$,  we have the following bounds:
	\begin{itemize}
			\item[(i)]  if $|V_{\xi}|\leq {d}/{2}$, then 
		\begin{align*}
			|\mathfrak{m}_{t}(\xi)-p_{\lambda}^{1}(\xi)|\lesssim\min\left\{\kappa(d,\lambda)\|\xi\|,\big(\kappa(d,\lambda)\|\xi\|\big)^{-1}\right\}+\kappa(d,\lambda)^{-1};
		\end{align*}
			\item[(ii)]  if $|V_{\xi}|>{d}/{2}$, then 
		\begin{align*}
			|\mathfrak{m}_{t}(\xi)-p_{\lambda}^{2}(\xi)|\lesssim\min\left\{\kappa(d,\lambda)\|\xi+{\bm{1}}/{2}\|,\big(\kappa(d,\lambda)\|\xi+ {\bm{1}}/{2}\|\big)^{-1}\right\}+\kappa(d,\lambda)^{-1}.
		\end{align*}
	\end{itemize}
\end{proposition}

\subsection{The large-scale case}
This subsection is devoted to estimating the
maximal functions corresponding to $\mathcal{A}_{t}^{d}$ with the supremum taken over the set $\mathcal{D}_{c_{3},\infty}$.
\begin{theorem}\label{0115.1}
	Let $c_{3}> 0$ and define $\mathcal{D}_{c_{3},\infty}=\left\{t \in\mathbb{D}:t\geq c_{3}d^{{3}/{2}}\right\}$. For every $f\in L_{2}(\mathcal{N})$, there exists  a constant $C> 0$ independent of the dimension $d$ such that 
	\begin{align*}
		\|\sup_{t\in  \mathcal{D}_{c_{3},\infty}}\mathcal{A}_{t}^{d}f\|_{L_{2}(\mathcal{N})}\leq C\|f\|_{L_{2}(\mathcal{N})}.
	\end{align*}
\end{theorem}
	%
For this proof, we  require    expansions for the multiplier $\mathfrak{m}_{t}$ given in (\ref{250423.1}).  For $q\geq 1$, $(p,q)=1$, $t>0$ and $\xi\in\mathbb{T}^d$, we denote
\begin{align*} 
	\mathfrak{a}_{t,{p}/{q}}(\xi)=\frac{\lambda^{{d}/{2}-1}}{2|\mathbb{S}_{t}\cap\mathbb{Z}^d|}e^{-2\pi \mathrm{i}\lambda{p}/{q}}G( {p}/{q};\llbracket q\xi \rrbracket )\widehat{\sigma}\big(t({\llbracket q\xi \rrbracket}/{q}-\xi)\big),
\end{align*}
where \begin{align*} 
	G({p}/{q};x) =q^{-d}\sum_{n\in \mathbb{N}_{q}^{d}}e^{2\pi \mathrm{i}(|n|^2{p}/{q}+ \langle x,n \rangle/{q})},\quad x\in\mathbb{Z}^d,
\end{align*}
is the  $d$-dimensional Gaussian sum.   The estimates for  $G({p}/{q};x)$  are  available (see \cite[Lemma 3.1]{mirek2023dimension}): 
   \begin{align}\label{1205.10}
	|G(p/q;x)|\leq (2/q)^{d/2},~~ x\in\mathbb{Z}^d, \quad\mbox{and}\quad \sum_{n\in \mathbb{N}_{q}^{d}}|G({p}/{q};n)|^2=1. 
	\end{align}
Given $N= \lfloor t \rfloor$ and consider the corresponding sequence
\begin{align*}
	H_{N}=\left\{{p}/{q}\in \mathbb{Q}:~0\leq p\leq q\leq N,~(p,q)=1\right\}.
\end{align*}
 For $1\leq n\leq N+1$, $t>0$ and $\xi\in\mathbb{T}^d$, define
\begin{align}\label{2485.15}
	\mathfrak{b}_{t,n}(\xi)=\frac{\lambda^{{d}/{2}-1}}{2|\mathbb{S}_{t}\cap\mathbb{Z}^d|}\sum_{{p}/{q}\in H_{N}\atop q\geq n}\sum_{x\in\mathbb{Z}^d}e^{-2\pi \mathrm{i}\lambda {p}/{q}}G ( {p}/{q};x)\Theta(q\xi-x)\widehat{\sigma}\big(t({x}/{q}-\xi)\big),
\end{align}
where 
	\begin{align}\label{250415.1}
	\Theta(\xi)=\prod_{j=1}^{d}\vartheta(\xi_{j}), \quad \xi=(\xi_{1},\cdots,\xi_{d})\in\mathbb{R}^d, 
\end{align}
and $\vartheta  \in C_{c}^{\infty}\big((- {1}/{4}, {1}/{4})\big)$ is a smooth, compactly supported function such that $\vartheta\equiv1 $  on $[-{1}/{8}, {1}/{8}]$ and    $\|\vartheta\|_{L_{\infty}(\mathbb{R})}\leq 1$.

\begin{proposition}[\cite{mirek2023dimension}]\label{1202.1}
	There exists a constant $C>0$ such that for all $d\geq 16$ and $t>0$ satisfying $t\geq Cd^{3/2}$ and all integers $1\leq n\leq N+1 $, 
	we have 
	\begin{align*}
		\mathfrak{m}_{t}(\xi)=\sum_{ {p}/{q}\in H_{N}\atop q<n}\mathfrak{a}_{t,{p}/{q}}(\xi)+	\mathfrak{b}_{t,n}(\xi)+E_{t,n}(\xi),
	\end{align*}
	where the error term $E_{t,n}(\xi)$ satisfies 
	\begin{align*}
		|E_{t,n}(\xi)|\lesssim^{d}  \frac{d^{{3d}/{4}}}{\lambda^{{d}/{4}\,-1}}, 
	\end{align*}
	uniformly for $\xi\in\mathbb{T}^d$, $1\leq n\leq N+1$, $d\geq 16$ and $t\geq Cd^{3/2}$.
\end{proposition}

\begin{lemma}\label{0109.6}
	There exists a constant $C>0$  such that for all integers $1\leq p\leq q$ with $(p,q)=1$ and $d\geq 16$, we have
	\begin{align*}
		\|\sup_{t\in  \mathcal{D}_{C,\infty}} \mathcal{F}^{-1}(\mathfrak{a}_{t,{p}/{q}}\widehat{f}\,)\|_{L_{2}(\mathcal{N})}\lesssim\|f\|_{L_{2}(\mathcal{N})},
	\end{align*}
	where the implicit constant is independent of $p,q$ and dimension $d$.
\end{lemma} 
\begin{proof}
	Fix $q\in \mathbb{N}$. For each  $w \in \mathbb{N}^{d}_{q}$, we denote   $T_{w} =\left\{ \xi\in\mathbb{T}^d: \llbracket q\xi \rrbracket\equiv w\,(\mbox{mod } q)\right\}$. Then any  $f\in L_{2}(\mathcal{N})$ admits the  decomposition:
	\begin{align*} 
		f=\sum_{w\in \mathbb{N}^{d}_{q}}f_{w},\mbox{ \quad with \quad }\widehat{f_{w}}=\widehat{f}\cdot\chi_{T_{w}},
	\end{align*}
where the functions $\widehat{f_{w}}$ have pairwise disjoint supports.   
By the triangle inequality,  (\ref{250413.1}), and the following estimate (see \cite[(3.19)]{mirek2023dimension}):
\begin{align}\label{1205.7}
	\frac{\lambda^{{d}/{2}-1}}{|\mathbb{S}_{t}\cap\mathbb{Z}^d|}\simeq\frac{1}{\sigma(\mathbb{S} )},\quad t\in  \mathcal{D}_{C,\infty},\,\, d\geq16,
\end{align}
 one has 
	\begin{align}\label{250328.1}
		 \|&\sup_{t\in \mathcal{D}_{C,\infty}} \mathcal{F}^{-1}(\mathfrak{a}_{t,{p}/{q}}\widehat{f}\,)\|_{L_{2}(\mathcal{N})}\nonumber\\
		\leq& \sum_{w\in \mathbb{N}^{d}_{q}}  \big\|\sup_{t\in  \mathcal{D}_{C,\infty}} \mathcal{F}^{-1}\big (e^{-2 \pi \mathrm{i}\lambda {p}/{q}}G ({p}/{q};\llbracket q\xi \rrbracket )\widehat{\mu}\big(t({\llbracket q\xi \rrbracket}/{q}-\xi)\big)\widehat{f_{w}}(\xi)\big ) \big\|_{L_{2}(\mathcal{N})}.
	\end{align}
Lemma \ref{0109.1} implies  that 
	\begin{align}\label{250328.2}
	 \mbox{RHS}(\ref{250328.1})\leq\sum_{w\in \mathbb{N}^{d}_{q}} \big|G ({p}/{q};w ) \big|\, \big\|\sup_{t\in  \mathcal{D}_{C,\infty}} \mathcal{F}^{-1}\big (\widehat{\mu}\big (t ({\llbracket q\xi \rrbracket}/{q}-\xi )\big )\widehat{f_{w}}(\xi)\big )\big\|_{L_{2}(\mathcal{N})}.
	\end{align}  
Recalling for $\omega\in \mathbb{N}^{d}_{q}$,  it was shown in \cite[(5.6)]{mirek2023dimension}  that 
\begin{align*}  
	\xi-{\llbracket q\xi \rrbracket}/{q}=\xi-{\omega}/{q}-\llbracket\xi-{\omega}/{q}\rrbracket,\quad\xi\in T_{\omega}.
\end{align*} 
Thus,
 \begin{align}\label{1205.8}
 	\big \|\sup_{t\in  \mathcal{D}_{C,\infty}} \mathcal{F}^{-1}\big (\widehat{\mu}&\big (t ( {\llbracket q\xi \rrbracket}/{q}-\xi )\big )\widehat{f_{w}}(\xi)\big )\big \|_{L_{2}(\mathcal{N})}\nonumber\\
 	=
	&\big \|\sup_{t\in  \mathcal{D}_{C,\infty}} \mathcal{F}^{-1}\big (\widehat{\mu} \big (t (\xi- {w}/{q}-  \llbracket\xi-{w}/{q}  \rrbracket)\big )\widehat{f_{w}}(\xi)\big )\big \|_{L_{2}(\mathcal{N})}\nonumber\\ 
=&\big \|\sup_{t\in \mathcal{D}_{C,\infty}} \mathcal{F}^{-1}\big (a_{t}(\xi)\widehat{f_{w}} (\xi+{w}/{q})\big )\big \|_{L_{2}(\mathcal{N})}\lesssim \|f_{w}\|_{L_{2}(\mathcal{N})},
\end{align}
 where the last inequality follows from Lemma \ref{1209.1}.
Finally, by the Cauchy-Schwarz inequality together with (\ref{1205.10}), (\ref{250328.1})-(\ref{1205.8}),  and the Plancherel theorem, one has
	\begin{align*}
		\|\sup_{t\in \mathcal{D}_{C,\infty}} \mathcal{F}^{-1}(\mathfrak{a}_{t,{p}/{q}}\widehat{f}\,)\|_{L_{2}(\mathcal{N})}	\lesssim \bigg (\sum_{w\in \mathbb{N}^{d}_{q}}\big |G ({p}/{q};w )\big |^2\bigg )^{{1}/{2}} \bigg(\sum_{w\in \mathbb{N}^{d}_{q}} \|f_{w}\|_{L_{2}(\mathcal{N})}^2\bigg)^{{1}/{2}}
		= \|f\|_{L_{2}(\mathcal{N})}.
	\end{align*}
\end{proof}
We now turn to the second main  ingredient to complete the proof of  Theorem \ref{0115.1}, which   bases on the  noncommutative transference principle, i.e., Proposition \ref{1206.1}.  

\begin{lemma}\label{1206.7}
	There exist  constants $C>0$ and $n_{0}\in \mathbb{N}$ such that for all $d\geq 16$, we have 
	\begin{align*}
		\|\sup_{t\in \mathcal{D}_{C,\infty}} \mathcal{F}^{-1}(\mathfrak{b}_{t,n_{0}}  \widehat{f}\,)\|_{L_{2}(\mathcal{N})}\lesssim \|f\|_{L_{2}(\mathcal{N})},
	\end{align*}
	where 	  the implicit constant is independent of  the dimension $d$.
\end{lemma}
\begin{proof}
For this proof, we first give some necessary notations.   Let $\phi\in C_{c}^{\infty}\big((-{1}/{2},\,{1}/{2})\big)$ be a smooth, compactly supported function such that $\phi\equiv1 $  on $[-{1}/{4},{1}/{4}]$ and    $\|\phi\|_{L_{\infty}(\mathbb{R})}\leq 1$; and denote
\begin{align*}
	\Phi(\xi)=\prod_{j=1}^{d}\phi(\xi_{j}),\quad \xi=(\xi_{1},\cdots,\xi_{d})\in\mathbb{R}^d.
\end{align*}
Set 
\begin{align*}
	H_{\infty}=\bigcup_{N\in\mathbb{N}}H_{N}=\left\{ {p}/{q}\in \mathbb{Q}:~1\leq p\leq q,~(p,q)=1\right\}.
\end{align*}
  For every $ {p}/{q}\in H_{\infty}$, define two auxiliary     operators  $U^{q}_{t},  V^{ {p}/{q}} :L_2(\mathcal{N})\rightarrow L_2(\mathcal{N})$ by setting
\begin{align*} 
	U^{q}_{t}f =\mathcal{F}^{-1}\bigg (\sum_{x\in\mathbb{Z}^d}\Theta_{(q)} (\xi-{x}/{q} )\widehat{\mu}\big (t (\xi-{x}/{q} )\big )\widehat{f}\,\bigg),\quad f\in L_2(\mathcal{N}),
\end{align*}	 
 	and    
\begin{align*}
	V^{{p}/{q}}f =\mathcal{F}^{-1}\bigg(\sum_{x\in\mathbb{Z}^d}G({p}/{q};x)\Phi_{(q)}(\xi-{x}/{q})\widehat{f}\,\bigg),\quad f\in L_2(\mathcal{N}),
\end{align*}
where $\Theta$ is given in  (\ref{250415.1}) and 
\begin{align*}
	\Theta_{(q)}(\xi)=\Theta(q\xi), \quad   \Phi_{(q)}(\xi)=\Phi(q\xi).
\end{align*}
Since $t\in  \mathcal{D}_{C,\infty}$, we have $t\geq Cd^{{3}/{2}}$ and $N=\lfloor t\rfloor\geq Cd^{{3}/{2}}-1$. Recalling $	\mathfrak{b}_{t,n_{0}}$ given in (\ref{2485.15}),   using Lemma \ref{0109.1}, (\ref{1205.7}), and the fact that  $\Theta \Phi = \Theta$,  we obtain
\begin{align*}
	&\|\sup_{t\in \mathcal{D}_{C,\infty}} \mathcal{F}^{-1} (	\mathfrak{b}_{t,n_{0}} \widehat{f}\,)\|_{L_{2}(\mathcal{{N}})}\\
	\lesssim& \sum_{ {p}/{q}\in H_{N}\atop q\geq n_{0}}\bigg \|\sup_{t\in\mathcal{D}_{C,\infty}} \mathcal{F}^{-1}\bigg (\sum_{x\in \mathbb{Z}^d} G ( {p}/{q};x )\Theta_{(q)}  (\xi- {x}/{q} )\widehat{\mu}\big (t ( {x}/{q}-\xi )\big )\widehat{f}\,\bigg )\bigg \|_{L_{2}(\mathcal{{N}})} \\
	\leq& \sum_{ {p}/{q}\in H_{\infty} \atop q\geq n_{0}}\|\sup_{t\in \mathcal{D}_{C,\infty}} U^{q}_{t}V^{ {p}/{q}}f\|_{L_{2}(\mathcal{{N}})},
\end{align*}
uniformly for $1 \leq n_{0} \leq C d^{3/2}$. The choice of $n_{0}$ will be specified later.

Once we establish that there exist  constants $D_{1}$, $D_{2}>0$ such that 
\begin{align}\label{1204.1} 
	\|\sup_{t\in \mathcal{D}_{C,\infty}}U^{q}_{t}f\|_{L_{2}(\mathcal{{N}})}\leq D_{1}^d\,\|f\|_{L_{2}(\mathcal{{N}})},
\end{align}
and
\begin{align}\label{1204.4}
	\|V^{{p}/{q}}f\|_{L_{2}(\mathcal{{N}})}\leq D_{2}^d\,q^{-{d}/{2}}\|f\|_{L_{2}(\mathcal{{N}})},
\end{align}
  uniformly for ${p}/{q}\in H_{\infty}$ and $d\geq 16$, this lemma will be   proved. Indeed,  taking $n_{0}=\lfloor (D_{1}D_{1})^{10}\rfloor+1$, we obtain
\begin{align*}
	\|\sup_{t\in \mathcal{D}_{C,\infty}} \mathcal{F}^{-1} (	\mathfrak{b}_{t,n_{0}} \widehat{f}\,)\|_{L_{2}(\mathcal{{N}})}
	\lesssim&\sum_{q\geq n_{0}} D_{1}^dD_{2}^d\,q^{-{d}/{2}+1}\|f\|_{L_{2}(\mathcal{{N}})}\\
	=& \sum_{q\geq n_{0}} q^{-{2d}/{5}+1}\bigg(\frac{(D_{1}D_{2})^{10}}{q}\bigg)^{{d}/{10}}\|f\|_{L_{2}(\mathcal{{N}})}\\
	 \leq&\sum_{q\geq n_{0}} q^{-{2d}/{5}+1}\bigg(\frac{(D_{1}D_{2})^{10}}{\lfloor (D_{1}D_{1})^{10}\rfloor+1}\bigg)^{{d}/{10}}\|f\|_{L_{2}(\mathcal{{N}})} 
	\lesssim \|f\|_{L_{2}(\mathcal{{N}})},
\end{align*}
as long as $d\geq 16$. Therefore, it  reduces to  showing $(\ref{1204.1})$ and $(\ref{1204.4})$, respectively.

	\vskip 0.2cm

	\textbf{Proof of  (\ref{1204.1})}: Let 
\begin{align*}
	m_{t}(\xi)=\Theta_{(q)}(\xi)\widehat{\mu}(t\xi),
\end{align*} 
and  consider  the corresponding multiplier $(m_{t})_{per}^{q}(\xi)$ defined in (\ref{250603.2}). Evidently, the operator  $U^{q}_{t}$  coincides with the multiplier operator $(T_{t})^{q}_{dis}$   given in (\ref{250516.1}).  Therefore, it suffices to show 
 \begin{align}\label{250520.1}
 		\|\sup_{t\in \mathcal{D}_{C,\infty}} (T_{	t})^{q}_{dis}f\|_{L_{2}(\mathcal{{N}})}\leq D_{1}^d\,\|f\|_{L_{2}(\mathcal{{N}})}.
 \end{align}
  Notice that $\mbox{supp}\,m_{t}\subset q^{-1}Q$ and the convolution kernel corresponding to multiplier $m_{t}$ is a real-valued function, since	 
 \begin{align*} 
 \widehat{\mu}(\xi)=\int_{\mathbb{S}}e^{-2\pi\mathrm{i}\langle x,\xi\rangle }d\mu(x)=\int_{\mathbb{S}}\prod_{j=1}^{d}\cos(2\pi \mathrm{i}x_{j}\xi_{j})d\mu(x),
 \end{align*}
  and  $\Theta$ are radial functions. On the other hand, Lemma \ref{1209.2} implies that 
  \begin{align*}
	\|\sup_{t\in \mathcal{D}_{C,\infty}} &T_{t}f\|_{L_{2}( L_{\infty}(\mathbb{R}^d)\overline{\otimes}\mathcal{M})}=\|\sup_{t\in \mathcal{D}_{C,\infty}}  \mathcal{F}^{-1}(\widehat{\mu}(t\cdot)\Theta_{(q)}\widehat{f})\|_{L_{2}( L_{\infty}(\mathbb{R}^d)\overline{\otimes}\mathcal{M})}\nonumber\\
	\lesssim&\|\Theta_{(q)}\widehat{f}\,\|_{L_{2}( L_{\infty}(\mathbb{R}^d)\overline{\otimes}\mathcal{M})}
	\leq\|\Theta\|_{ L_{\infty}(\mathbb{R}^d)}\|f\|_{L_{2}( L_{\infty}(\mathbb{R}^d)\overline{\otimes}\mathcal{M})}\leq\|f\|_{L_{2}( L_{\infty}(\mathbb{R}^d)\overline{\otimes}\mathcal{M})}.
\end{align*}  
 Therefore, applying Proposition \ref{1206.1} to our setting, (\ref{250520.1})  is proved  with $D_{1}=3$.


\vspace{0.1cm}

	\textbf{Proof of  (\ref{1204.4})}:
 We invoke (\ref{250411.1})   with  
  $m(\xi)= \sum_{x\in\mathbb{Z}^d}G({p}/{q};x)\Phi_{(q)}(\xi-{x}/{q})$.   Notice that $\mbox{supp}\, \Phi_{(q)} \subset q^{-1}Q$. It was established in    \cite[Page 24]{mirek2023dimension} that $\sum_{n\in\mathbb{Z}^d}|\widehat{\Phi_{(q)}}(n)|\leq A^d$, where $A$ is a universal constant depending only on $\Phi$.  On the other hand,   $G(p/q;x)$ is  $q\mathbb{Z}^d$ periodic and  satisfies $\sup_{x\in\mathbb{Z}^d}|G(p/q;x)|\leq (2/q)^{d/2}$ by (\ref{1205.10}).   Therefore, the application of (\ref{250411.1})  yields (\ref{1204.4}) with $D_{2}=\sqrt{2}A$.
\end{proof}
Now, we have all ingredients to prove Theorem  \ref{0115.1}.

\noindent\textbf{Proof of Theorem \ref{0115.1}}: Let $c_{3}>0$ be a sufficiently large universal constant to ensure that the conclusions of Proposition \ref{1202.1}, Lemma \ref{0109.6} and Lemma \ref{1206.7} are satisfied. Choose $n_{0}\in\mathbb{N}$
   large enough to satisfy the conclusion of Lemma \ref{1206.7}.  Then, by applying Proposition \ref{1202.1} together with Lemma \ref{0109.6}, Lemma \ref{1206.7}, Proposition \ref{8.10.1}, and Plancherel’s theorem, we obtain
\begin{align*}
	\|\sup_{t\in  \mathcal{D}_{c_{3},\infty}} \mathcal{A}_{t}^{d}&f\|_{L_{2}(\mathcal{N})}\\ 
	\leq&\bigg\|\sup_{t\in  \mathcal{D}_{c_{3},\infty}} \mathcal{F}^{-1}\bigg(\sum_{ {p}/{q}\in H_{N},\atop q<n_{0}}\mathfrak{a}_{t, {p}/{q}}\widehat{f}\,\bigg)\bigg\|_{L_{2}(\mathcal{N})}+\|\sup_{t\in \mathcal{D}_{c_{3},\infty}} \mathcal{F}^{-1}(\mathfrak{b}_{t, n_{0}}\widehat{f}\,)\|_{L_{2}(\mathcal{N})}\\
	&+\|\sup_{t\in\mathcal{D}_{c_{3},\infty}} \mathcal{F}^{-1}(E_{t,n_{0}}\widehat{f}\,)\|_{L_{2}(\mathcal{N})}\\
	\lesssim &\bigg(1+ \sum_{ t\in \mathbb{D} \atop t>c_{3}d^{3/2}} C^d \frac{d^{{3d}/{4}}}{t^{ {d}/{2}-2}}\bigg)\|f\|_{L_{2}(\mathcal{N})}\\
= &\bigg(1+ \sum_{ t\in \mathbb{D} \atop t>c_{3}d^{3/2}} \frac{d^9C^{12}}{t^4}\bigg( \frac{d^{3}C^4}{t^2}\bigg)^{ {d}/{4}-3}\bigg)\|f\|_{L_{2}(\mathcal{N})}.
\end{align*}
Since $d\geq 16$, taking in
a constant $c_{3}>0$ such that $c_{3}^2>2C^4$ we obtain
\begin{align*}
	\|\sup_{t\in  \mathcal{D}_{c_{3},\infty}} \mathcal{A}_{t}^{d}f\|_{L_{2}(\mathcal{N})}\lesssim  \bigg(1+\sum_{ t\in \mathbb{D} \atop t>c_{3}d^{3/2}}t^{-4}\bigg )\|f\|_{L_{2}(\mathcal{N})}\lesssim\|f\|_{L_{2}(\mathcal{N})},
\end{align*}
which completes the proof.
$ \hfill\square $

\vspace{0.5cm}

\noindent\textbf{\small Acknowledgements} {\small This work  is supported by the National Natural Science Foundation of China (No. 12371138)}.
  
\vspace{0.2cm}
 
\noindent\textbf{\small Data availability} Data sharing is not applicable to this article, as no datasets were generated or analyzed during the current study.

\end{document}